\documentclass[nohyperref]{article}

\usepackage{microtype}
\usepackage{graphicx}
\usepackage{subfigure}
\usepackage{booktabs} 

\usepackage[colorlinks=true,citecolor=blue]{hyperref}

 \usepackage[accepted]{icml2022}

\usepackage{amsmath}
\usepackage{amssymb}
\usepackage{mathtools}
\usepackage{amsthm} 
\usepackage[framemethod=TikZ]{mdframed}
\usepackage{lipsum}
\usepackage{paralist}
\usepackage{hyperref}

\usepackage{cleveref}
\usepackage{nicefrac}

\usepackage{color,colortbl}
\definecolor{darkblue}{rgb}{0.0,0.0,0.55}
\definecolor{darkred}{rgb}{0.5,0.0,0.0}
\hypersetup{
  colorlinks = true,
  citecolor  = darkblue,
  linkcolor  = darkred,
  filecolor  = darkblue,
  urlcolor   = darkblue,
}

\numberwithin{equation}{section}
\usepackage{thm-restate}
\declaretheorem[name=Corollary]{corollary}
\declaretheorem[name=Lemma]{lemma}
\declaretheorem[name=Proposition]{proposition}
\declaretheorem[name=Remark, style=remark]{remark}
\declaretheorem[name=Theorem]{theorem}
\declaretheorem[name=Definition]{definition}
\declaretheorem[name=Example]{example}
\declaretheorem[name=Experiment]{experiment}

\declaretheorem[name=Assumption]{assumption}

\declaretheorem[name=Fact]{fact}
\declaretheorem[name=Takeaway]{takeaway}

\usepackage{natbib}
\usepackage{algorithm}
\usepackage{color}
\usepackage{mathrsfs}
\usepackage{enumitem}
\usepackage{bm}

\usepackage{multirow}
\usepackage{booktabs}
\usepackage{makecell}
\usepackage{graphicx}
\usepackage{subfigure}
\usepackage{bbm}
\usepackage{caption}
\usepackage{thmtools}
\usepackage{thm-restate}
\usepackage{bbold} 

\usepackage{float} 
\usepackage{mdframed}

\newcommand{\D}{\mathrm{d}}
\newcommand{\vect}[1]{\ensuremath{\mathbf{#1}}}

\newcommand{\inp}[2]{\left\langle #1,~#2 \right\rangle}

\newcommand{\norm}[1]{\left\|{#1} \right\|}

\newcommand{\R}{\mathbb{R}} 

\newcommand{\E}{\mathbb{E}}

\newcommand{\vv}{\vect{v}}
\newcommand{\vx}{\vect{x}}

\newcommand{\vq}{\vect{q}}

\newcommand{\zz}[1]{{\boldsymbol\theta}^{#1}}

\newcommand{\lmax}[1]{\lambda_{\max}^{#1}}

\newcommand{\sm}[1]{L_{#1}}

\newcommand{\pp}{{\bm{p}}}
\newcommand{\vze}{{\bm{p}}}

\newcommand{\eps}{{\epsilon}}

\newcommand{\Xc}{\mathcal{X}}
\newcommand{\Nc}{\mathcal{N}}
\newcommand{\Sta}{\mathcal{S}}
\newcommand{\Stap}{\mathcal{C}}

\newcommand{\ur}{{the unstable regime}} 

\newcommand{\rp}{{relative progress}} 

\newcommand{\uc}{{unstable convergence}}

\newcommand{\loc}{W^{\sf s\oplus c}_{\sf loc}}

\newcommand{\dir}[2]{L(#2;#1)}
\newcommand{\rr}[1]{{\sf RP}(#1)}
\newcommand{\qqq}{\qed}

\definecolor{MITBrown}{RGB}{164, 31, 50} 
\newcommand\miteecs[1]{\textbf{\textcolor{MITBrown}{#1}}}

\usepackage[textsize=tiny]{todonotes}

\icmltitlerunning{Unstable Convergence of GD}

\begin{document}

\twocolumn[
\icmltitle{Understanding the Unstable Convergence of Gradient Descent}



\icmlsetsymbol{equal}{*}

\begin{icmlauthorlist}
\icmlauthor{Kwangjun Ahn}{mit,simons}
\icmlauthor{Jingzhao Zhang}{tsing}
\icmlauthor{Suvrit Sra}{mit}
\end{icmlauthorlist}

\icmlaffiliation{simons}{Part of this work was done while Kwangjun Ahn was visiting the Simons Institute for the Theory of Computing, Berkeley, CA, USA.}

\icmlaffiliation{mit}{{\bf Department of EECS, MIT}, Cambridge, MA, USA}

\icmlaffiliation{tsing}{{\bf IIIS, Tsinghua University}, Beijing, China}

\icmlcorrespondingauthor{Kwangjun Ahn}{kjahn@mit.edu}

\icmlkeywords{Machine Learning, ICML}

\vskip 0.3in
]



\printAffiliationsAndNotice{ } 

\begin{abstract}

Most existing analyses of (stochastic) gradient descent rely on the condition that for $L$-smooth costs, the step size is less than $2/L$.  However, many works have observed that in machine learning applications step sizes often do not fulfill this condition, yet (stochastic) gradient descent still converges, albeit in an unstable manner. We investigate this unstable convergence phenomenon from first principles, and discuss key causes behind it. We also identify its main characteristics, and how they interrelate based on both theory and experiments, offering a principled view toward understanding the phenomenon.
\end{abstract}

 \section{Introduction}
  
Gradient descent (GD) runs the iteration
\begin{align*} 
  \zz{t+1} = \zz{t} - \eta \nabla f(\zz{t})\,,
\end{align*}
seeking to optimize a cost function $f$. It also provides a conceptual foundation for stochastic gradient descent (SGD), one of the key algorithms in modern machine learning. A vast body of literature that analyzes (S)GD assumes that the cost $f$ is $L$-smooth (we say $f$ is $L$-smooth if $\|\nabla f(\zz{})-\nabla f(\zz{'})\|\le L\|\zz{ }-\zz{'}\|$ for all $\zz{},\zz{'}$), and subsequently exploits the associated ``\emph{descent lemma}'': 
 \begin{align} \label{descent}
     f(\zz{t+1}) \leq f(\zz{t}) -\eta \bigl(1 - L\frac{\eta}{2}\bigr)  \norm{\nabla f(\zz{t})}^2\,.
 \end{align} 
To ensure descent via inequality~\eqref{descent}, the condition
\begin{align}\label{cond:step}
  L < \frac{2}{\eta},
\end{align}
is imposed. This condition ensure that GD decreases the cost $f$ at each iteration.
Whenever condition~\eqref{cond:step} holds, we call it the \miteecs{``stable'' regime} in this paper.

When the cost is quadratic, condition~\eqref{cond:step} is in fact necessary for stablility: if $\eta > \frac{2}{L}$, then GD diverges (see~\autoref{fact:1}). This observation carries over to most convex optimization settings and also neural networks when using the neural tangent kernel approximations \cite{jacot2018neural,li2018learning,lee2019wide}. Thus, it is reasonable to assume condition~\eqref{cond:step} for those analyses. However, for general nonconvex costs, it is not clear whether the stable regime condition \eqref{cond:step} is required or even reasonable.

\begin{figure} 
	\centering
	{
		\includegraphics[width=0.49\columnwidth]{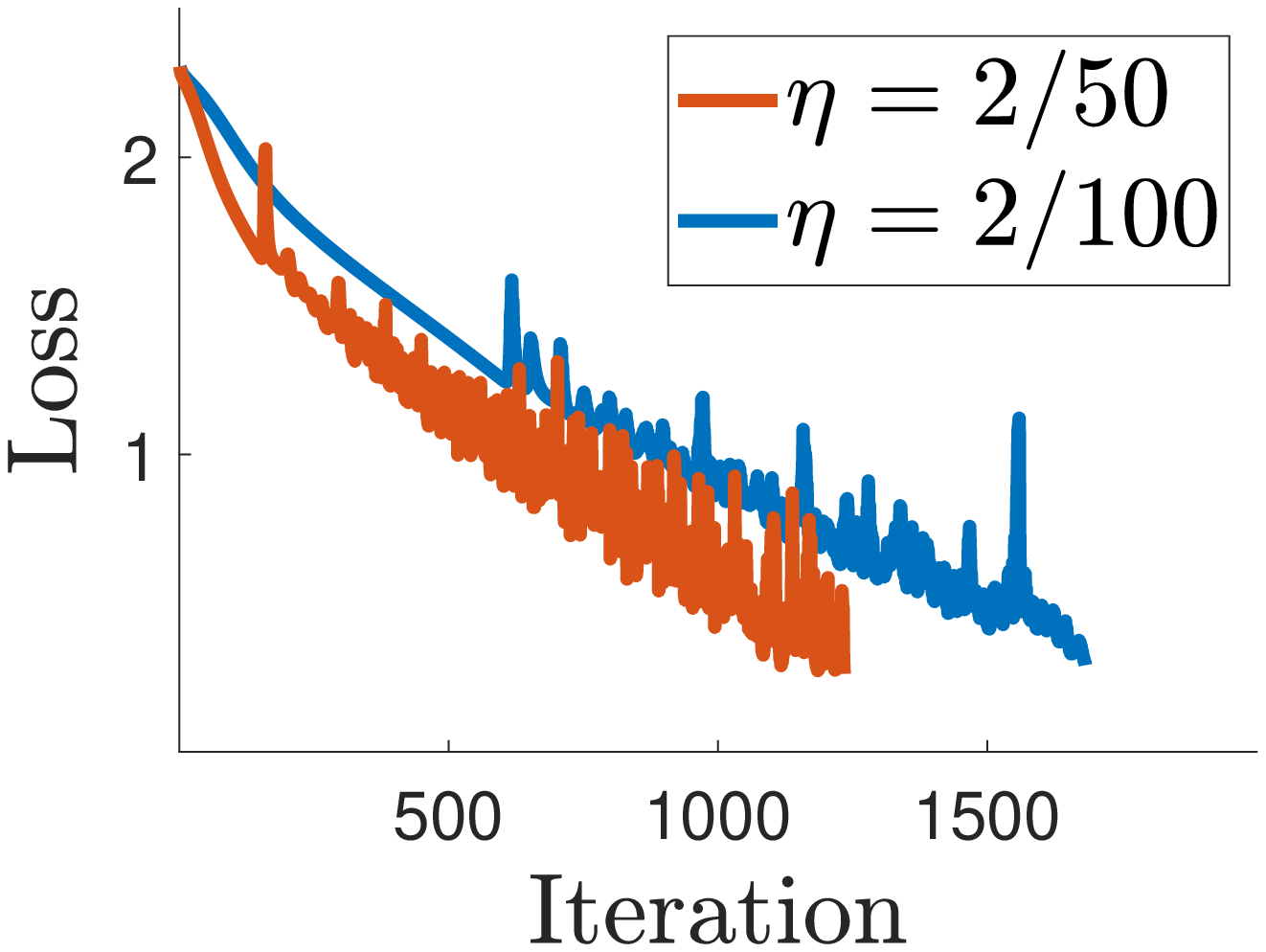} 	\includegraphics[width=0.49\columnwidth]{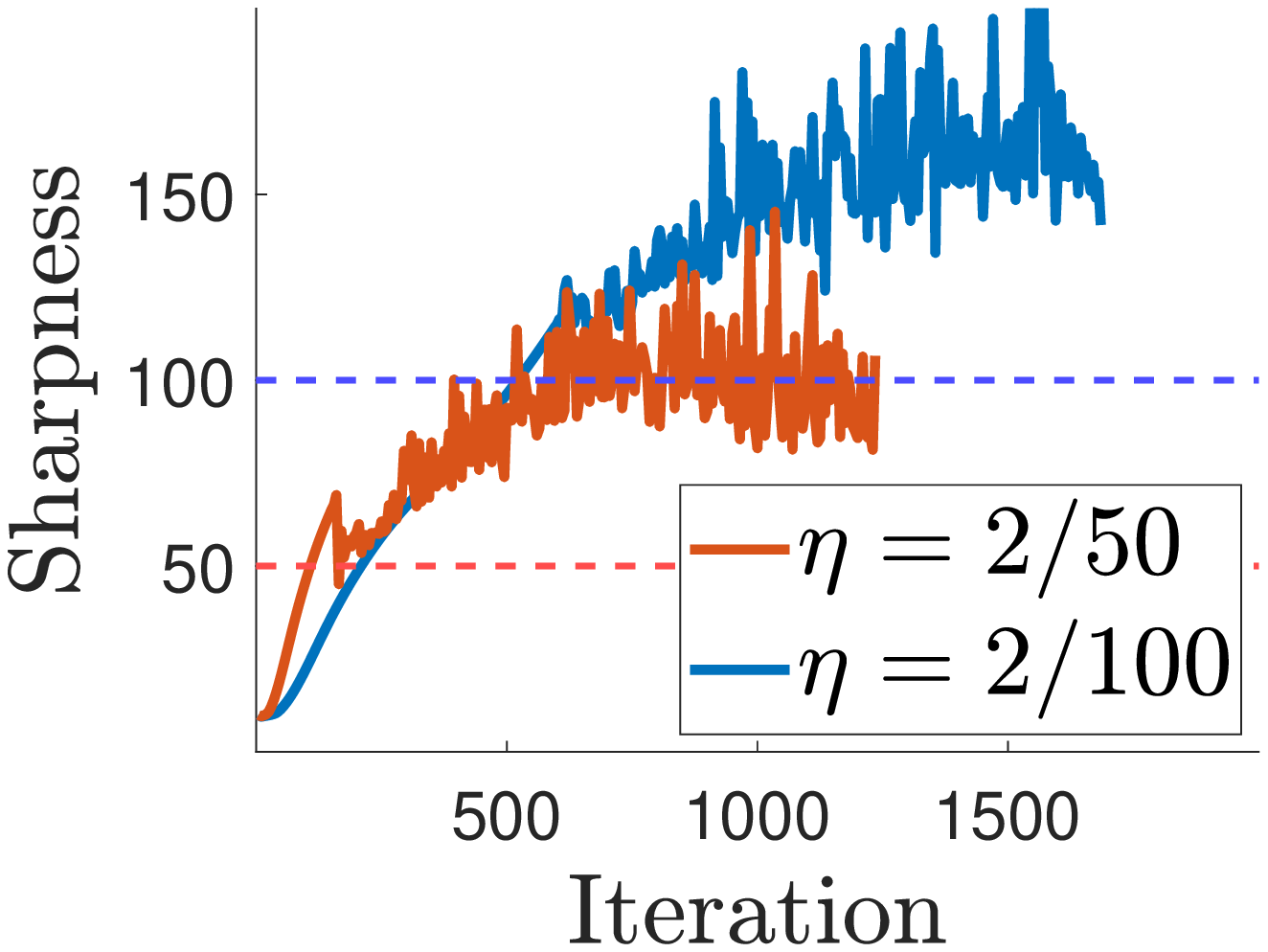}
	}  
	\caption{{\bf Example of \uc{}} for training CIFAR-10 with GD. We follow the experimental setup of \cite{Cohen2021}; see \autoref{ex:cifar} for details. We use a ReLU network. Here, condition \eqref{cond:step} fails, but the training loss still (non-monotonically) decreases in the long run.}
	\label{fig:1}
\end{figure}

Recently, it has been observed that GD on neural networks often violates condition~\eqref{cond:step}. More specifically, \citet{Cohen2021} observe that when we run GD to train a neural network, the condition \eqref{cond:step} fails, but contrary to the common wisdom from convex optimization, the training loss still (non-monotonically) decreases in the long run. See Figure~\ref{fig:1} for an example of this phenomenon. We call this phenomenon \miteecs{``unstable'' convergence}.
 
Unfortunately, very little is known about unstable convergence. The causes and implications of this phenomenon have not been explored in the literature. More importantly, the main features as well as the scope of this phenomenon have not been discussed. Characterizing the main features is important because it not only furnishes better understanding of this bizarre phenomenon, but also lays a foundation for future theoretical studies; especially, the main characteristics of this phenomenon will help build a more practical theory of the neural network optimization.

{\bf Contributions.} In light of the above motivation, the main contributions of this paper are as follows:
\begin{enumerate}
  \vspace*{-7pt}
  \setlength{\itemsep}{-1pt}
\item  We discuss the main causes driving the unstable convergence phenomenon (\autoref{sec:cause}).  
   
\item  We identify the main features that characterize unstable convergence in terms of how loss, iterates, and sharpness\footnote{In this paper, following \cite{Cohen2021},  sharpness means the  maximum eigenvalue of the  loss Hessian, i.e., $\lmax{}({\nabla^2 f(\zz{t}))}$.} evolve with GD updates. Moreover, we investigate and clarify the relations between them.
Our characterizations demonstrate that the features of unstable convergence are in stark contrast with those of traditional stable convergence, suggesting that their optimization mechanisms are significantly different.
 
 \item In particular, the main features considered in this work provide alternative ways to identify unstable convergence in practice.
      
\end{enumerate} \autoref{fig:summary} provides a more technical overview of our main findings, along with their interpretations.

 \begin{figure}
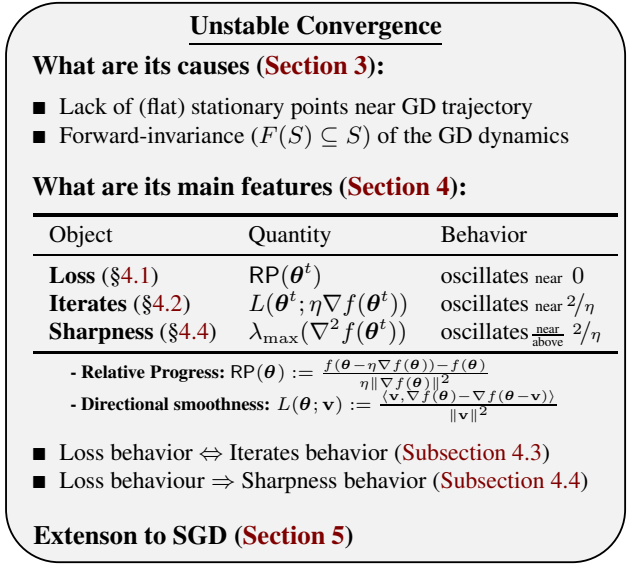

     \mdfsetup{%
   middlelinewidth=.7pt,
   backgroundcolor=gray!10,
   roundcorner=20pt}
 \begin{mdframed}

\begin{center}
    \underline{\bf Unstable Convergence}
\end{center}
       \vspace{3pt}
       
       {\bf What are its causes (\autoref{sec:cause}):}
        \vspace{-3pt}
       \begin{small}
 \begin{list}{{\tiny $\blacksquare$}}{\leftmargin=0.5em}  \setlength{\itemsep}{-3pt}  
           \item Lack of (flat) stationary points near GD trajectory
        \item Forward-invariance ($F(S)\subseteq S)$ of the GD dynamics
       \end{list}
       \end{small}

       {\bf What are its main features (\autoref{sec:feature}):}
       \vspace{-10pt}
      \begin{center}
         \begin{small} 
\begin{tabular}{lll}
\toprule
Object & Quantity & Behavior   \\
\midrule
{\bf Loss} (\S\ref{sec:loss})  & $\rr{\zz{t}}$ & {\footnotesize oscillates}  {\tiny near } $0$ \\
{\bf Iterates} (\S\ref{sec:iterate}) & $\dir{\eta \nabla f(\zz{t})}{\zz{t}}$ & {\footnotesize oscillates}  {\tiny  near }$\nicefrac{2}{\eta}$   \\
{\bf Sharpness} (\S\ref{sec:sharp}) & $\lmax{}(\nabla^2 f(\zz{t}))$ &   {\footnotesize oscillates}{\tiny   $\frac{\text{near}}{\text{above}}$}  $\nicefrac{2}{\eta}$\\ 
\bottomrule
\end{tabular} 
\end{small}
\end{center}

       \vspace*{-3pt}
     
  \begin{scriptsize}
 \quad\quad{\bf  - Relative Progress:} $\rr{\zz{}}:=\frac{f(\zz{}-\eta \nabla f(\zz{}) )-f(\zz{})}{\eta \norm{\nabla f(\zz{})}^2}$

\quad\quad{\bf  -   Directional smoothness:}     $\dir{\vv}{\zz{}}:=\frac{\langle\vv,\nabla f(\zz{})- \nabla f(\zz{} -  \vv )\rangle}{\norm{\vv}^2}$
  \end{scriptsize}
  
  \vspace{-2pt}
  
  \begin{small}
 
 \begin{list}{{\tiny $\blacksquare$}}{\leftmargin=0.5em}
  \setlength{\itemsep}{-3pt}         
         \item  Loss behavior $\Leftrightarrow$ Iterates behavior (\autoref{sec:equiv}) 
         
   \item Loss behaviour   $\Rightarrow$ Sharpness behavior (\autoref{sec:sharp}) 
    
  \end{list}
  \end{small}
  {\bf Extenson to SGD (\autoref{sec:sgd})}
  
   \end{mdframed}
   \caption{Overview/summary of results.}\label{fig:summary}
 \end{figure}

\subsection{Related Work}
 
Under various contexts, several recents works have observed the unstable convergence phenomenon in training neural networks with (S)GD~\cite{wu2018sgd,xing2018walk,lewkowycz2020large,jastrzkebski2017three,jastrzkebski2018relation}. We refer readers to the related work section of \citet{Cohen2021} for greater context.

The unstable convergence phenomenon is first formally identified by \citet{Cohen2021}, and in their paper it is named \emph{edge of stability}.
More specifically, they observe a more refined version of the unstable convergence: when training a neural network with GD, the sharpness at the iterate goes beyond the threshold $\nicefrac{\eta}{2}$, and often saturates right at (or above) the threshold.
In \autoref{sec:sgd}, we will explore the relations between our main features and their observed phenomenon.

{\bf Concurrent works.} Recently,
\citet{ma2022multiscale} also investigate the causes of unstable convergence based on their empirical observations. 
Their main observation is that  unstable convergence might be due to the landscape of loss function where the loss grows slower than a quadratic near the local minima.
As we will see in \autoref{sec:cause_convergence}, their main finding is consistent with our explanation.
They also demonstrate through examples that such ``sub-quadratic'' growth near the minima is caused by the heterogeneity of data; see their Section 6 for details.

Another work by \citet{arora2022understanding} identifies a setting in which one can prove the unstable convergence phenomenon theoretically. 
More specifically, they show that the normalized gradient descent dynamics of form $\zz{t+1} = \zz{t} - \eta  \nabla f(\zz{t})/\norm{\nabla f(\zz{t})}$ can provably exhibit the unstable convergence phenomenon near the  minima under some suitable assumptions; see their Section 4 for details.

 \section{Unstable GD Can't Reach Stationary Points} \label{sec:suggest}

In this section, we build intuitions about what the unstable regime $\eta >\nicefrac{2}{L}$ suggests. 
First, note that the fixed points $\zz{\infty}$ of the GD dynamics $\zz{t+1} = \zz{t} - \eta \nabla f(\zz{t})$ are the stationary points, i.e., points such that $\nabla f(\zz{\infty})=0$.  
Hence, the GD dynamics will eventually approach one of the stationary points, and in order to understand the unstable regime, we first need to understand the behavior of the dynamics near the stationary points whose sharpness is greater than $\nicefrac{2}{\eta}$.

As a warm up, we first consider the simplest setting of quadratic costs where the sharpness is constant globally. 
We begin with the following well-known fact.

 \begin{fact} \label{fact:1}
 On a quadratic cost $f(\zz{})= \frac{1}{2}\zz{\top} P \zz{} +\vq^\top \zz{} +r$, GD  will diverge if 
any eigenvalue of $P$ exceeds the threshold $2/\eta$.  For convex quadratics,  this condition is ``iff.''\qqq
 \end{fact}
 
Below we quickly illustrate this fact through an example.

\begin{example}\label{ex:simplequad}
Consider optimizing a quadratic cost $f(\theta_1,\theta_2) = 20 \theta_1^2+ \theta_2^2$. Note that in this case $L=40$. 
Let us run GD on this cost with $\eta = 2/39,\ 2/40,\ 2/41$.

{\centering	
		\includegraphics[width=0.25\columnwidth]{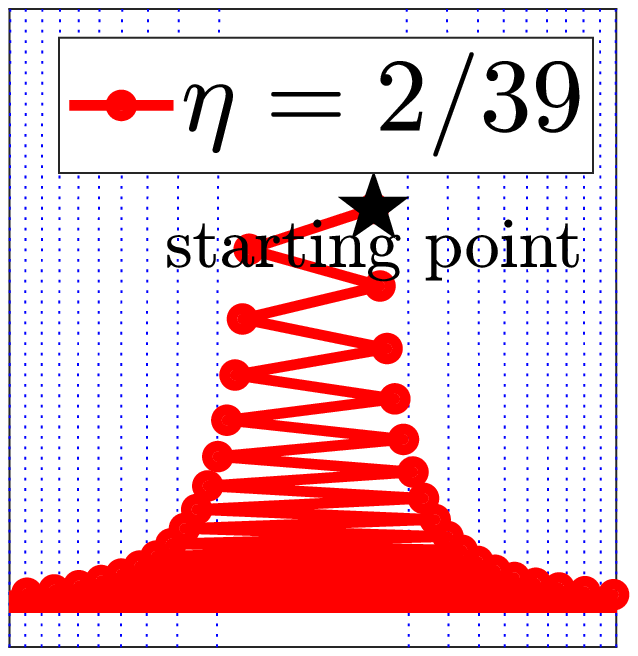}	\includegraphics[width=0.25\columnwidth]{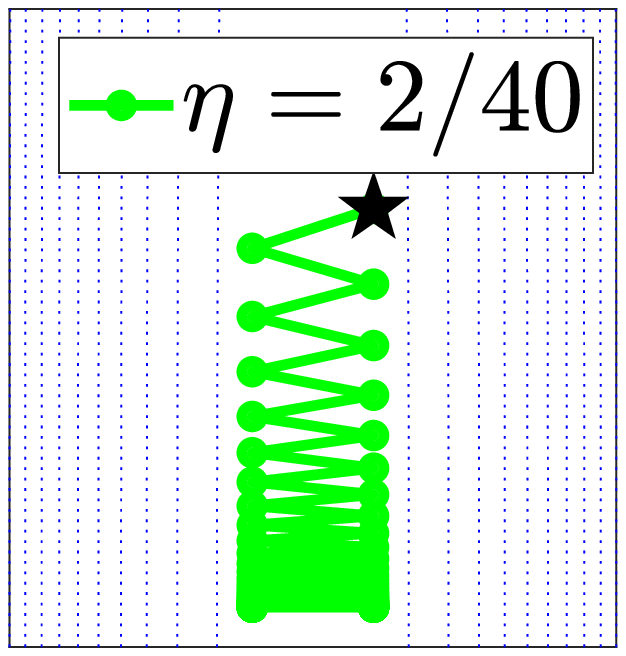}
		\includegraphics[width=0.25\columnwidth]{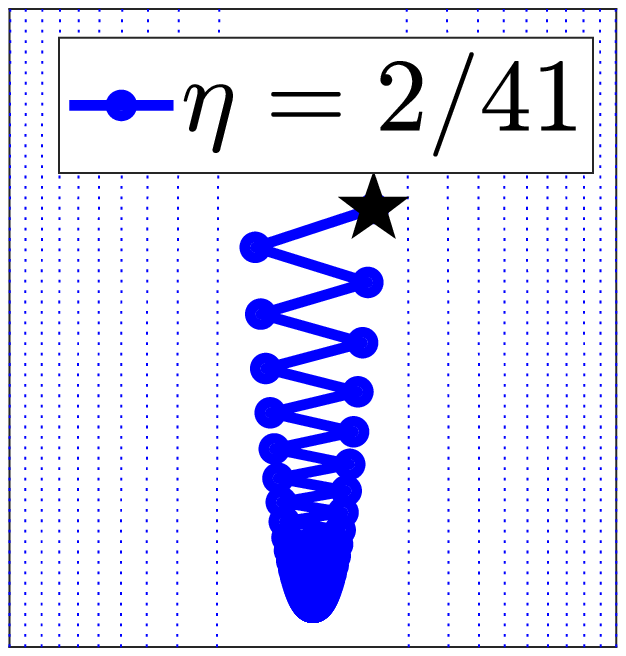}

	}
 
 As shown in the above plots, GD converges to the optimum if $\eta <2/L$ and it diverges  if $\eta >2/L$. \qqq
\end{example}

 Due to the above fact, one can build the following intuition: 
 when a stationary point has sharpness greater than $\nicefrac{2}{\eta}$, the GD dynamics cannot converge to the stationary point.
 
 We first formalize this intuition.
 In particular, we show that 
 GD cannot converge to a stationary point that has sharpness greater than $2/\eta$.
We make the following assumptions  about the spectrum of Hessian at stationary points and the non-degeneracy of the GD dynamics.

\begin{assumption}\label{as:gen}
Let $F(\zz{})= \zz{}- \eta \nabla f(\zz{})$ and assume that for any subset $S$ of measure zero, $F^{-1}(S)$ is of measure zero.
Moreover, each  stationary point $\pp$  satisfies $\frac{1}{\eta} \not\in \lambda (\nabla^2 f(\pp))$, where $\lambda (\nabla^2 f(\pp))$ denotes the set of eigenvalues of the Hessian of $f$ at $\pp$. 
\end{assumption}

\begin{theorem} \label{thm:master}
For a given subset $\Xc$ of the domain of parameter $\zz{}$, assume  that $f$ is $C^2$ in $\Xc$. 
Suppose that for each stationary point $\pp\in \Xc$, it holds that either $\lambda_{\min}(\nabla^2 f(\pp))<0$  or  $\lambda _{\max}(\nabla^2 f(\pp)) > \frac{2}{\eta}$.
Then under  \autoref{as:gen}, 
 there is a measure-zero subset $\Nc$ s.t. for all initializations $\zz{0}\in \Xc\setminus \Nc$, the GD dynamics $\zz{t+1} = \zz{t} - \eta \nabla f(\zz{t})$  do not converge to any of the stationary points in $\Xc$.
\end{theorem}
\begin{proof}[{\bf Proof of \autoref{thm:master}}]
The proof is inspired by those of \cite{lee2016gradient,panageas2017gradient}.
First, we recall Stable Manifold Theorem~\citep[ Thm III.7]{shub2013global}.
   \begin{lemma}  \label{thm:stable}
 Let $\vze$ be a fixed point for the $C^r$ local diffeomorphism $h: U\to \R^n$ where $U$ is an open neighborhood of $\vze$ in $\R^n$ and $r\geq 1$.
    Let $E^{\sf s} \oplus E^{\sf c} \oplus E^{\sf u}$ be the invariant spliiting of $\R^n$ into the generalized eigenspaces of $Dh(\vze)$ corresponding to eigenvalues of absolute value less than one, equal to one, and greater than one. 
    To the $Dh(\vze)$-invariant subspace $E^{\sf s} \oplus E^{\sf c}$ there is an associated local $h$-invariant $C^r$ embedded disc $\loc$ of dimension $\dim(E^{\sf s} \oplus E^{\sf c})$ and ball $B$ around $\vze$ such that $ h(\loc)\cap B \subset \loc$. Moreover, if $h^n(\vx)\in B$ for all $n\geq 0$, then $\vx \in \loc$. 
    \end{lemma}
To apply \autoref{thm:stable}, we first show that the GD dynamics $F$  is a local diffeomorphism at each stationary point $\pp$ satisfying $\frac{1}{\eta}\notin \lambda(\nabla^2 f(\pp))$. This follows from the inverse function theorem: (i) Note that $F$ is a $C^1$ vector field since $f$ is $C^2$, (ii) the Jacobian of $F$ is equal to $DF(\pp) = I-\eta \nabla^2f(\pp)$, and since $\frac{1}{\eta}\notin \lambda(\nabla^2 f(\pp))$, the Jacobian is invertible.
Hence, by inverse function theorem, we conclude that $F$ is a local diffeomorphism around $\pp$. 

Hence for each stationary point $\pp$ satisfying $1/\eta\notin \lambda(\nabla^2 f(\pp))$, we can apply \autoref{thm:stable} at $\pp$. Let $B_\pp$ be the open ball due to \autoref{thm:stable}.
Let $\Sta$ be the set of stationary points. Consider the following open cover
    \begin{align} \label{cover}
       \Stap :=\bigcup_{\substack{\pp: \text{ stationary point}\\
       \frac{1}{\eta}\notin \lambda(\nabla^2 f(\pp))}} B_\pp\,.
    \end{align} 
    Then from \autoref{as:gen}, it follows that $\Sta \subset \Stap$ and hence $\Stap$ is an open cover of $\Sta$.
    Thus, Lindel\"of's lemma guarantees that there exists a countable subcover of $\Stap$, i.e., there exist $\pp_1,\pp_2,\dots$ s.t. $\Stap=\cup_{i=1}^\infty B_{\pp_i}$.  If GD converges to a stationary point $\pp$, there must exist $t_0$ and $i$ such that $F^{t}(\pp_0)\in B_{\pp_i}$ for all $t\geq t_0$. 
    From \autoref{thm:stable}, we conclude that $F^{t}(\zz{}_0)\in\loc(\pp_i)$.
  In other words,  we have $\zz{}_0 \in F^{-t}(\loc(\pp_i))$ for all $t\geq t_0$.  Hence the set of initial points in $\Xc$ for which GD converges to a stationary point is a subset of 
 \begin{align*}
     \Nc := \bigcup_{i=1}^\infty \bigcup_{t=0}^\infty F^{-t}(\loc(\pp_i)) \,.
 \end{align*}
  Now from the assumption that either  $\lambda_{\min}(\nabla^2 f(\pp))<0$  or  $\lambda _{\max}(\nabla^2 f(\pp)) > \frac{2}{\eta} $, 
 it follows that $I-\eta \nabla^2 f(\pp)$ has an eigenvalue whose absolute value is greater than $1$.
 Hence, for each stationary point $\pp$, $\dim( E^{\sf u})\geq 1$.
   This implies that each $\loc(\pp)$ has measure zero, and from the assumption that $F^{-1}(\Xc)$ is of measure zero if $\Xc$ is of measure zero, it holds that each $F^{-t}(\loc(\pp_i))$ is of measure zero.
   Thus, being a countable union of measure zero sets, $\Nc$ is measure zero.
   It follows that for initialization $\zz{}_0 \in \Xc\setminus\Nc$, the GD dynamics never converge to a stationary point in $\Xc$.  \end{proof}

\begin{remark}
Note that \autoref{thm:master} applies to the case when  stationary points are not isolated. 
Moreover, the condition that every stationary point $\pp$  satisfies $\frac{1}{\eta} \not\in \lambda (\nabla^2 f(\pp))$ can be relaxed to the condition that the open cover $\Stap$ in \eqref{cover} covers the entire set of stationary points $\Sta$. 
\end{remark}

The main takeaway of  \autoref{thm:master} is that for almost all initializations, GD cannot converge to the stationary point whose sharpness is larger than $2/\eta$ even when there is only a single eigenvector whose eigenvalue exceeds the threshold $2/\eta$.
Having this intuition, we next discuss how ``convergence'' could happen under the unstable regime.

\section{How Can Unstable GD ``Converge''?} \label{sec:cause}
 
In the previous section, we saw that when stationary points have large sharpness relative to the step size, GD cannot converge to those stationary points.
However, as we saw in \autoref{fig:1},  GD can still ``converge'' under the unstable regime; GD still manages to (non-monotonically) decrease the training loss in the long run.
In this section, we understand this bizzare co-occurrence.
We first discuss some possible causes for the unstable regime.

\subsection{What Causes the Unstable Regime}
\label{sec:cause_instability}

 One possible cause for the unstable regime is that the landscape has only ``trivial'' stationary points; we will formalize the meaning of ``trivial'' shortly. 
 This situation turns out to be quite common for neural networks as illustrated by the following result. 

\begin{proposition}\label{prop:trivial}
Assume the loss of neural network parametrized $\zz{}$ contains a weight decay term as follows,
\begin{align*}
    \ell(\zz{}) = \frac{1}{n}\sum_{i=1}^n f(\boldsymbol{x}_i, \zz{}) + \gamma \|\zz{}\|^2_2.
\end{align*}
If we partition the network parameter $\zz{} = [\boldsymbol{\xi}; \boldsymbol{\zeta}]$ such that a subset of the network parameters $\boldsymbol{\zeta}$ is positive homogeneous, i.e. for any input data $x_i$ and positive number $c > 0$
\begin{align*}
    f(\boldsymbol{x}_i, [\boldsymbol{\xi}, \boldsymbol{\zeta}]) =  f(x_i, [\boldsymbol{\xi}, c\boldsymbol{\zeta}]),
\end{align*}
Then the loss $\ell(\zz{})$ has no stationary point if $\boldsymbol{\zeta} \neq 0$.
\end{proposition}

\begin{proof}[{\bf Proof of \autoref{prop:trivial}}]
This statement follows by a simple observation that from positive homogeneity, 
$$\left \langle \nabla_{\boldsymbol{\zeta}} f(\boldsymbol{x}_i, [\boldsymbol{\xi}, \boldsymbol{\zeta}]), \boldsymbol{\zeta} \right\rangle = 0. $$
Therefore, if $\nabla_{\boldsymbol{\zeta}} \gamma \|\zz{}\|^2_2 \neq 0$, we have  
$$ \nabla_{\boldsymbol{\zeta}} f(\boldsymbol{x}_i, [\boldsymbol{\xi}, \boldsymbol{\zeta}]) + \nabla_{\boldsymbol{\zeta}} \gamma \|\zz{}\|^2_2 \neq 0,$$
which concludes the proof.
\end{proof}

Notice that the positive homogeneity parameters exist in many networks such as ResNet or Transformer when normalization layers exist ($\boldsymbol{a}_{L+1} = \boldsymbol{a}_L / \|\boldsymbol{a}_L\| $, where $\boldsymbol{a}_L$ denote the input to layer $L$, and $\| \cdot \|$ denotes a norm of choice).

Note that in practical settings, \autoref{prop:trivial} also suggests that there could be lack of flat minima near the GD trajectory. In the above example of ResNet or Transformer, the  networks often add a small $\epsilon$ term to the normalization $\boldsymbol{a}_{L+1} = \boldsymbol{a}_L / (\epsilon + \|\boldsymbol{a}_L\|) $ to avoid the loss being undefined at $\boldsymbol{a}_L = 0$. However, the stationary points only exist when $\|\boldsymbol{a}_L\| \approx \epsilon$, in which case the sharpness of the stationary point is very large (on the order of $\sim 1/\eps$). 

In fact, it has been extensively observed in the literature that the sharpness around GD with practical stepsize choices often goes beyond the threshold $2/\eta$.
This claim is verified through a comprehensive set of experiments and called
{\bf progressive sharpening} in \cite{Cohen2021}; we refer readers to their Section 3.1 for details.
For instance, the sharpness curve in \autoref{fig:1} shows this phenomenon.
Moreover, a similar phenomenon was observed in  \cite{wu2018sgd}, and they speculated that the density of sharp minima is much larger than the density of flat minima in the neural network landscape. See their Section 4.1 for details.
 
 We summarize our discussion regarding the causes of unstable regime as follows.

\begin{takeaway}
For practical stepisze choices, lack of (flat) non-trivial stationary points near the GD trajectory can cause GD to enter the unstable regime.
\end{takeaway}

 \subsection{Causes for Convergence}
 \label{sec:cause_convergence} 
 
 As we discussed in \autoref{fact:1}, for quadratic costs (or more generally for most convex costs), GD being in the unstable regime implies that GD will diverge entirely. 
 However, as demonstrated by \cite{Cohen2021} through a comprehensive set of experiments, in neural network training, this situation no longer holds.
 In this section, we discuss how in the unstable regime ``convergence'' could happen through examples.
As a warm-up, let us revisit the quadratic cost considered in \autoref{fact:1}, but this time with some modifications.

\begin{example}[``Flattened'' quadratic cost]  \label{ex:falt}
For the same quadratic cost as in \autoref{fact:1},  we chose the same diverging step size $\eta = 2/39 > 2/L$, but this time we change the cost a bit by applying $\tanh(\cdot )$ on top of the quadratic cost.
More formally, we consider the cost $  \tanh(20\cdot \theta_1^2 + \theta_2^2)$.
Due to the fact $\tanh\approx x$ near zero, this transformation wouldn't change the geometry near the global minimum.
We run GD on the modified cost, and the result looks as follows (we include the result for the original quadratic cost on the left for comparison):
	
{\centering	
		\includegraphics[width=0.49\columnwidth]{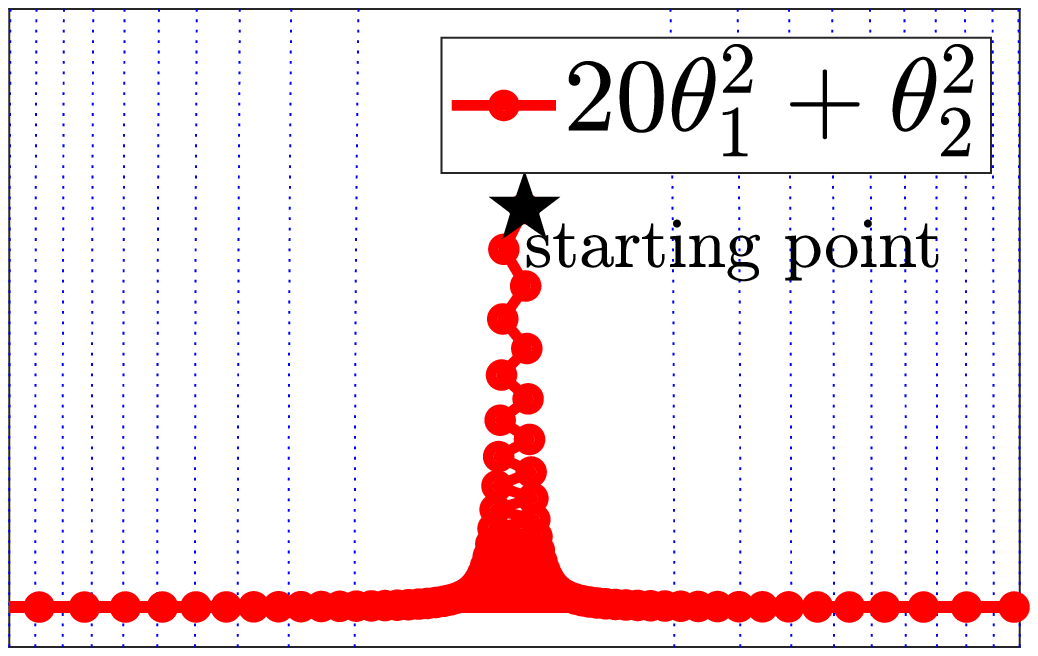} 	\includegraphics[width=0.49\columnwidth]{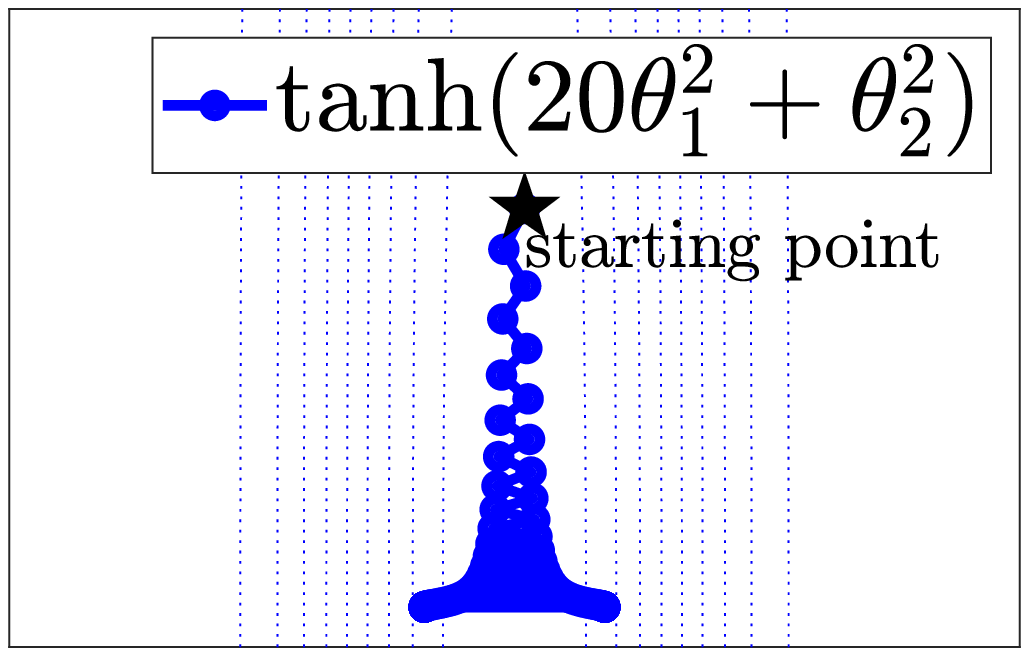}
	}
 
As one can see from the above plot, for the transformed cost, GD does not diverge in the unstable regime.
\qqq
\end{example} 
 
 The above toy example illustrates that indeed for nonconvex costs, being in the unstable regime does not necessarily imply complete divergence. 
  For the above example, this was possible because of $\tanh(\cdot)$, which `flattens'' out the landscape of the quadratic cost away from the minimum.

 More formally, let us denote the GD dynamics by $F(\zz{}):= \zz{}- \eta \nabla f(\zz{})$.
 Then the role of $\tanh(\cdot)$ in the above example is that it creates a compact subset near the minimum that is \emph{forward-invariant}: we say $S$ is forward-invariant with respect to the dynamics $F$ if $F(S)\subseteq S$. 
 Because the gradient of $\tan(\text{quadratic})$ vanishes as the point gets farther away from the minimum, there exists a forward-invariant compact subset $\Xc$ near the minimum.  
  
  \begin{remark}
 In a very recent concurrent work by \citep{ma2022multiscale}, this phenomenon is discussed in a more principled manner using the \emph{subquadratic growth property}.
 More specifically, they observed that for practical neural network settings, the loss landscape near the minima exhibits growth that is slower that quadratic, in which case the GD dynamics do not diverge entirely even in the unstable regime.
 See their Section 4 for details. 
  \end{remark}

 We demonstrate this point for neural network examples. 
 We first consider the simplest neural network example, namely a single hidden neuron network.

\begin{example}[Single neuron networks] \label{ex:tanh_single}
We consider a trivial task of fitting the data $(\textcolor{black}{1},\textcolor{black}{0})$ with a single hidden neuron neural network.
Formally, we consider two types of networks:
\begin{compactitem}
     \item linear network:  $f(\theta_1,\theta_2) = (\theta_1\cdot (\textcolor{black}{1}\cdot\theta_2)-\textcolor{black}{0})^2$.
     \item $\tanh$ network:  $f(\theta_1,\theta_2) = (\theta_1\cdot \tanh(\textcolor{black}{1}\cdot\theta_2)-\textcolor{black}{0})^2$.
\end{compactitem} 
We initialize both networks at  	$\zz{0} =(13, 0.01)$ choose step size $\eta=2/150$ to train them.

{\centering

\includegraphics[width=0.49\columnwidth]{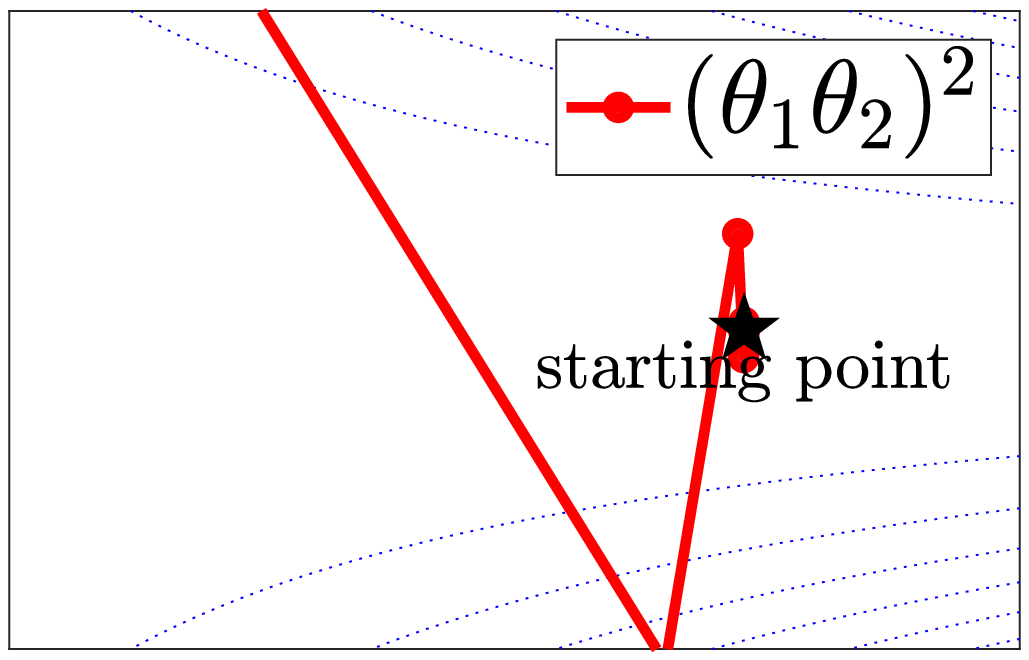} 	\includegraphics[width=0.49\columnwidth]{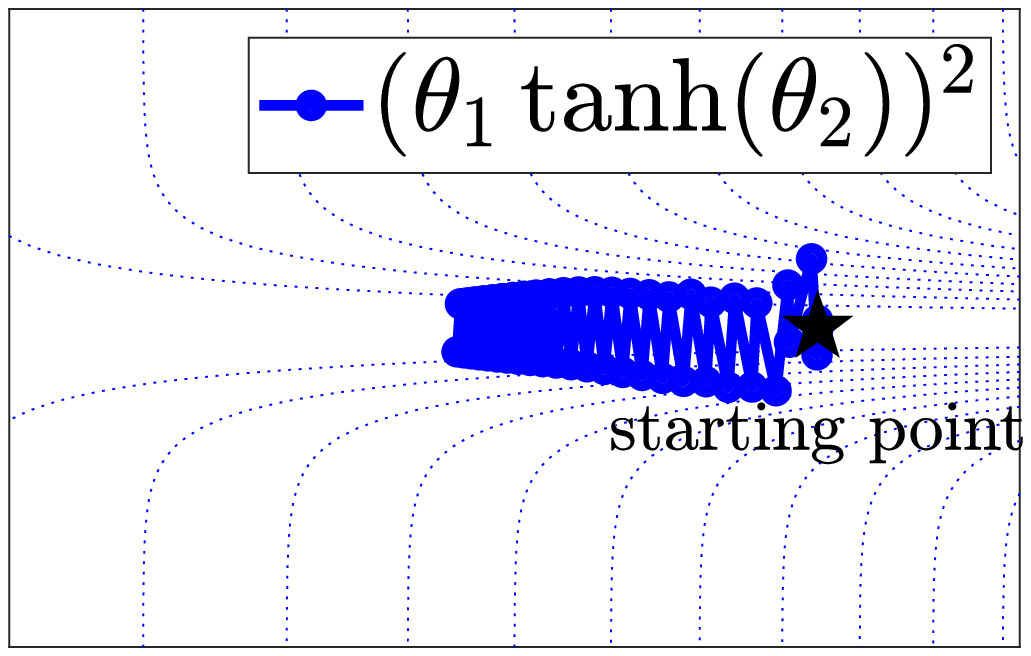}

 }
 
 As one can see from the above plots, for a linear network, the iterate quickly diverges, while for the $\tanh$ network, the iterate does not diverge and converges to a minimum (whose sharpness is indeed approximately equal to
 $2/\eta$). 
 \qqq
\end{example} 

\autoref{ex:tanh_single} illustrates that the use of activation function like $\tanh$ can create a compact forward-invariant subset near the minima, which helps GD not diverge in the unstable regime. 
In fact, the above example suggests that GD indeed exhibits some convergence behaviour where while being in the unstable regime, GD travels along the valley of minima until it finds a flat enough minimum where it can stabilize.

We now consider  more practical neural network examples inspired by the settings considered in \cite{Cohen2021}.

\begin{experiment}[CIFAR-10 experiment] \label{ex:cifar}
For this example, we follow the setting of the main experiment  \cite{Cohen2021} in their Section 3.
Specifically, we use (full-batch) GD to train
a neural network on  $5,000$ examples from CIFAR-10 with the CrossEntropy loss, and the network is a fully-connected architecture with two hidden layers
of width $200$.
Under this common setting, we consider three types of networks: (i) linear network without activations; (ii) $\tanh$ activations; (ii) ReLU activations.
We choose the step size $\eta=2/30$ and the results are as follows:

   {\centering	
		\includegraphics[width=0.49\columnwidth]{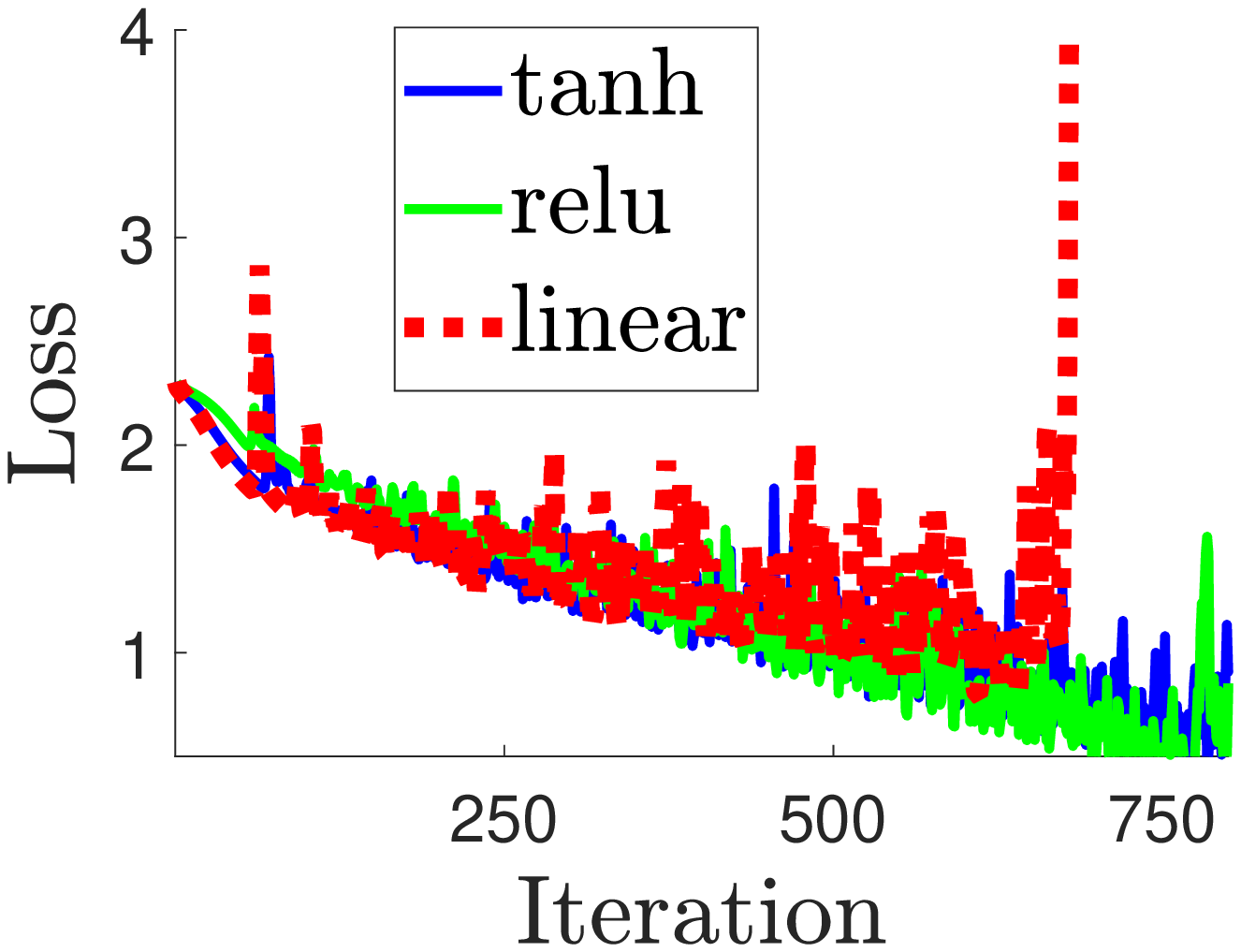} 	\includegraphics[width=0.49\columnwidth]{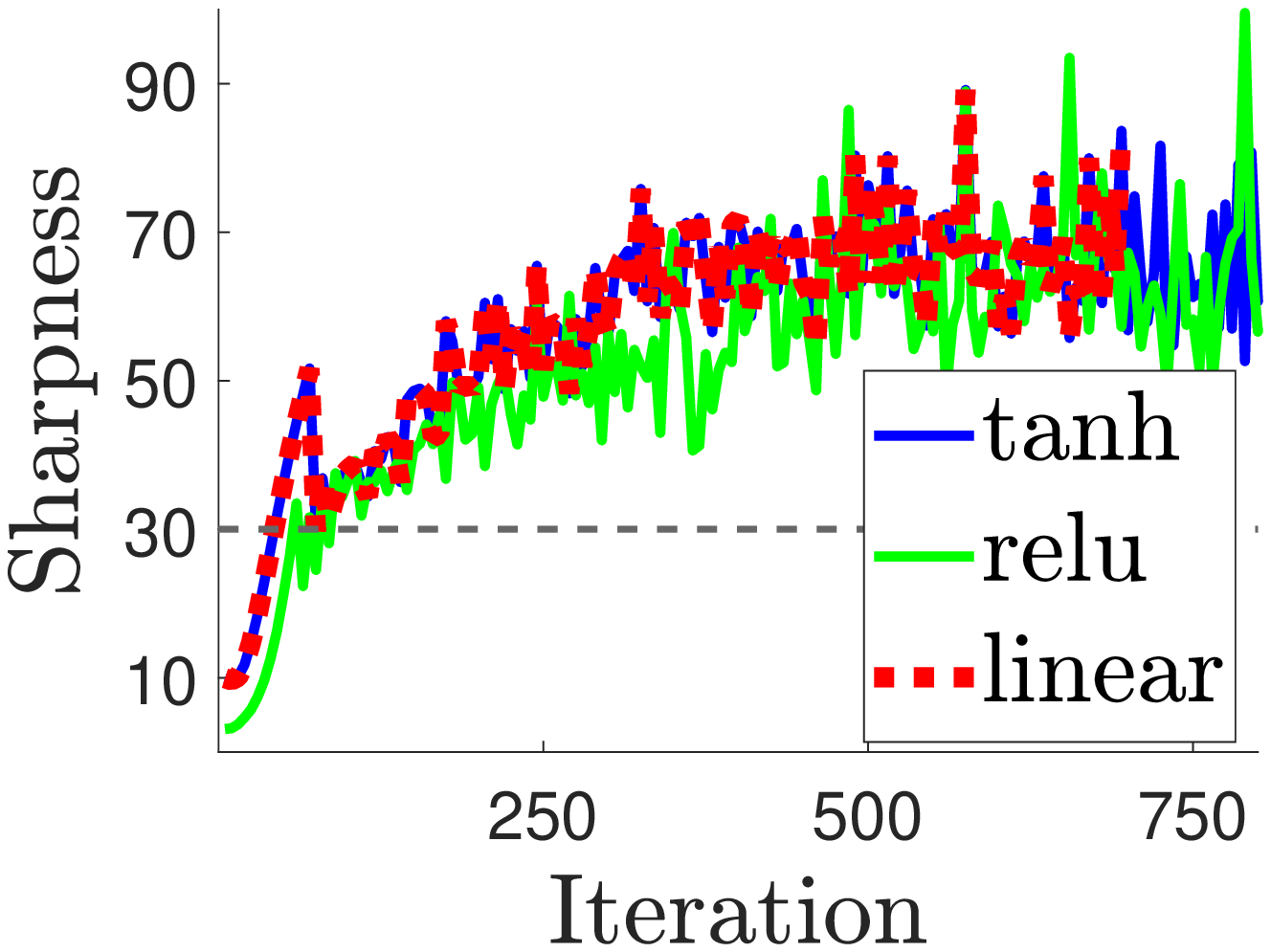}
		
	}
 
 As one can see from the above plot, GD converges for the networks with activation functions, while GD diverges without activation functions. 
 \qqq
\end{experiment}

We summarize our discussion regarding the causes for convergence as follows.

\begin{takeaway}
Ingredients of neural networks such as activation functions create a compact forward-invariant set near the minima, which helps GD (non-monotonically) converge in the unstable regime. 
\end{takeaway}

In this section, we have discussed the causes of unstable convergence and explain how the intuitions differ from those of  conventional convex optimization.
We next move on to study the main characteristics of unstable convergence.
For instance, we observe that under the unstable convergence phenomenon, the loss is very non-monotonic.
Can we understand the behavior of loss in a more principled way? 
 
 \section{Characteristics of the Unstable Convergence}
 \label{sec:feature}
 
 In this section,  we aim to quantify \uc{} through several  quantities that can be computed during the training.
 In particular, we will characterize the unstable convergence in terms of the loss behavior and the iterate behavior. We will later demonstrate that the two different behaviors are interconnected with each other.

 \subsection{Characteristics in Loss Behavior}
 
\label{sec:loss}
We first investigate what happens to the loss under unstable convergence.
 As a warm-up, we first consider the loss behavior under {\bf stable} convergence.
 
 \subsubsection{Warm-up: The Stable Regime}

Recall from the descent lemma \eqref{descent} that when GD is in the stable regime, then we have $f(\zz{t+1})-f(\zz{t}) \leq  - c \eta  \norm{\nabla f(\zz{t})}^2$ for some constant $c>0$.
Putting it differently, we have 
\begin{align*}
\frac{f(\zz{t+1})-f(\zz{t})}{\eta \norm{\nabla f(\zz{t})}^2} \leq -\mathrm{const.}
\end{align*}
Let us give the ratio on the LHS a name for convenience:
\begin{definition}[Relative progress ratio] We define \label{def:rp}
\begin{align*}
    \rr{\zz{}}:=\frac{f(\zz{}-\eta \nabla f(\zz{}) )-f(\zz{})}{\eta \norm{\nabla f(\zz{})}^2}\,.
\end{align*}
\end{definition}

Let us revisit \autoref{ex:cifar} and verify that for smaller step sizes the \rp{} ratio is indeed a negative number.

\begin{experiment}[CIFAR-10; stable regime] \label{exp:stable}
We use the same setting as \autoref{ex:cifar}, which follows the setting of the main experiment in \cite{Cohen2021}. For activations, we choose $\tanh$ following  \cite{Cohen2021}.
We choose much  smaller step sizes  so that GD is in the stable regime.
We plot the loss and the \rp{} ratio until the training accuracy hits $95\%$.

{\centering 

\includegraphics[width=0.45\columnwidth]{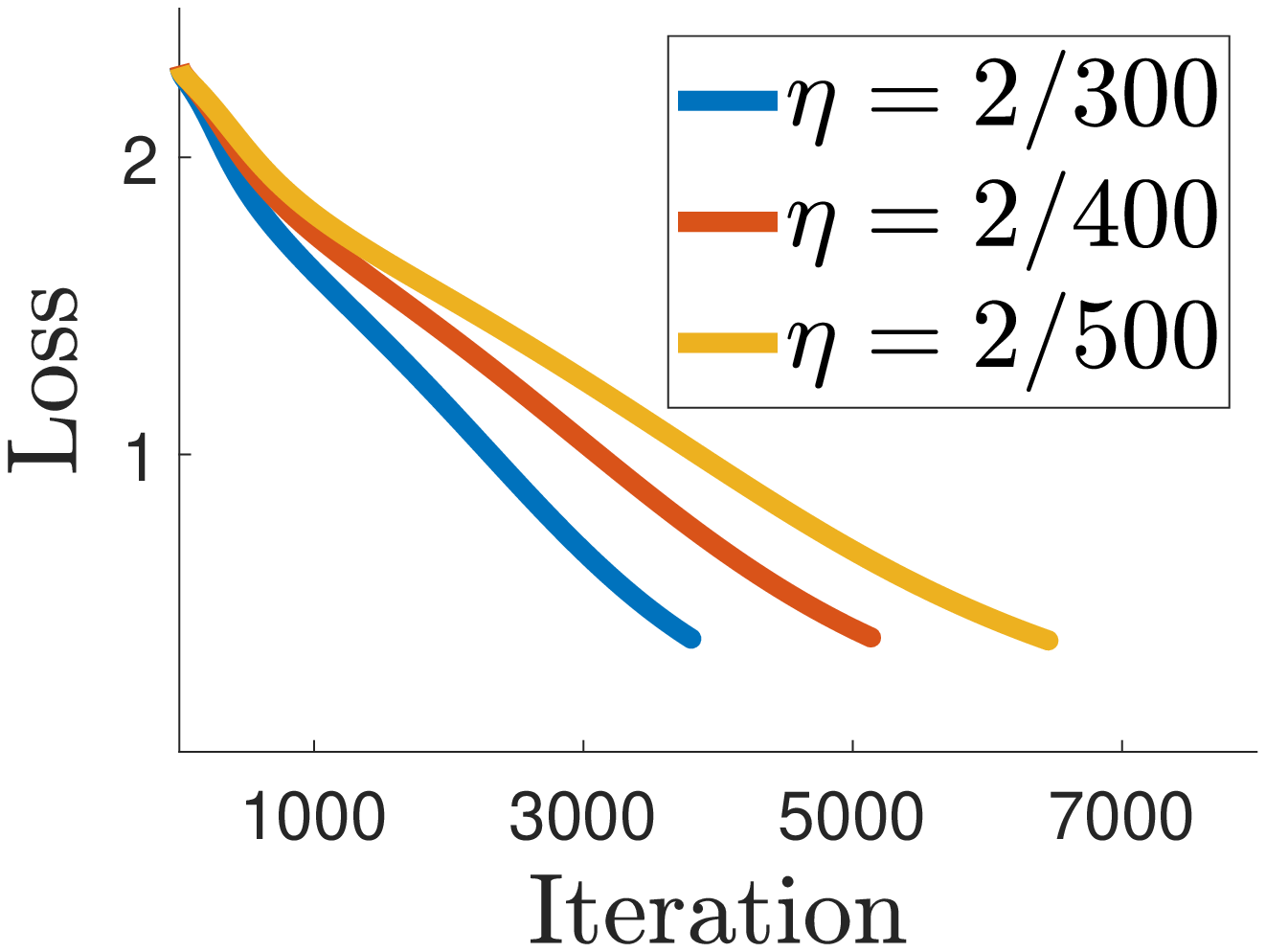} 	\includegraphics[width=0.45\columnwidth]{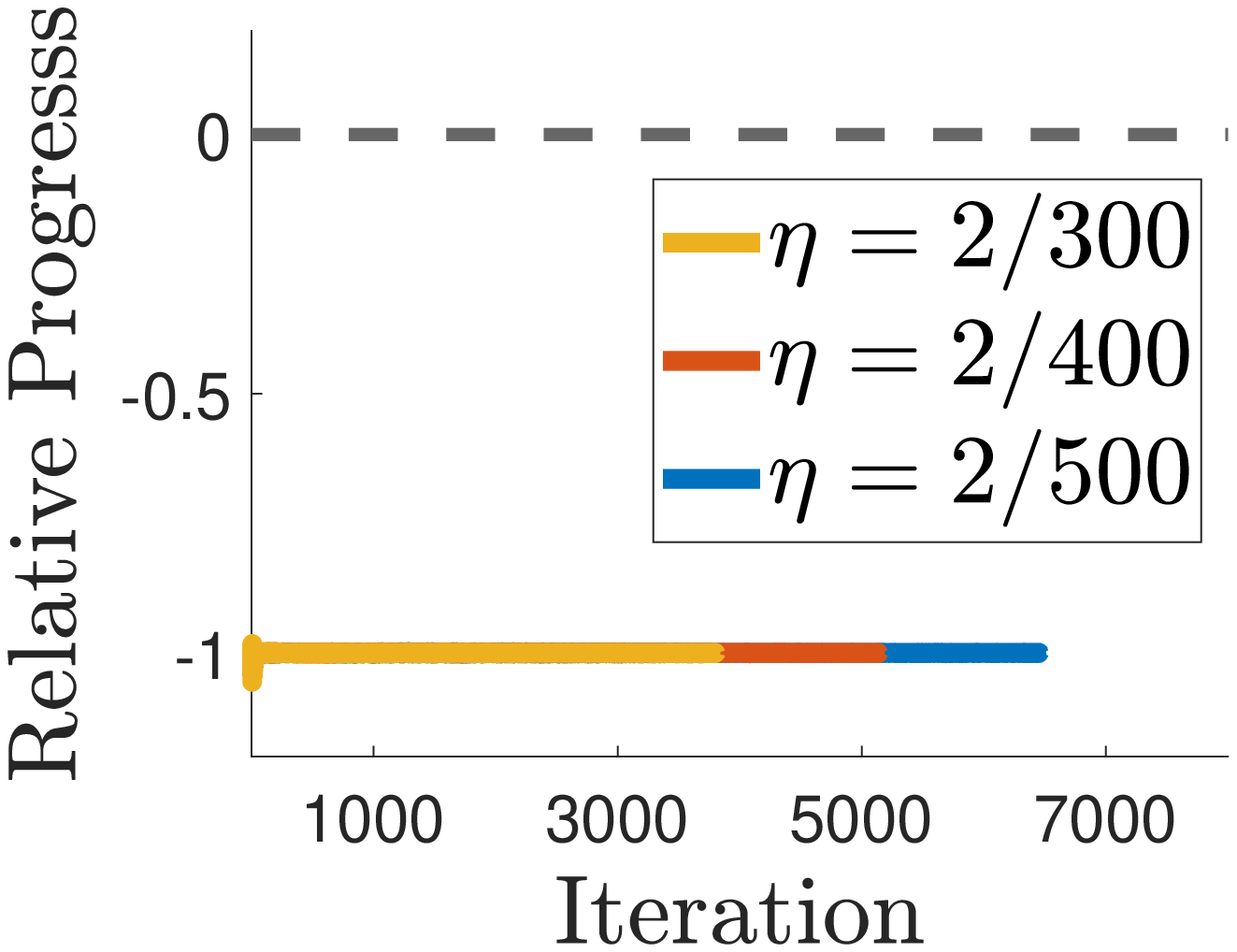}

 }
 
 From the above plots, one can see that  the \rp{} ratio stays negative for all iterations. 
 Moreover, there is no non-monotonic behavior in the loss curve.
 \qqq
\end{experiment}

\begin{remark} \label{rmk:stable-1}
Given the result above, one might wonder why the \rp{} saturates around $-1$.
Although we do not have a clear explanation, we suspect that this happens because the trajectory of GD quickly converges to a single direction. We will quickly revisit this later this section. See \autoref{rmk:stable}.
\end{remark}

\subsubsection{Relative Progress Ratio under Unstable Convergence}

Given that \rp{} is strictly negative number in the stable regime, we now investigate how \rp{} ratio behaves in the case of unstable convergence. 
 
\begin{experiment}[CIFAR-10; unstable regime] \label{exp:unstable}
We use the same setting as \autoref{exp:stable}.
This time we choose
step sizes larger so that GD operates in the unstable convergence regime.
We plot the loss and the \rp{} ratio until the training accuracy hits $95\%$.

{\centering 

\includegraphics[width=0.32\columnwidth]{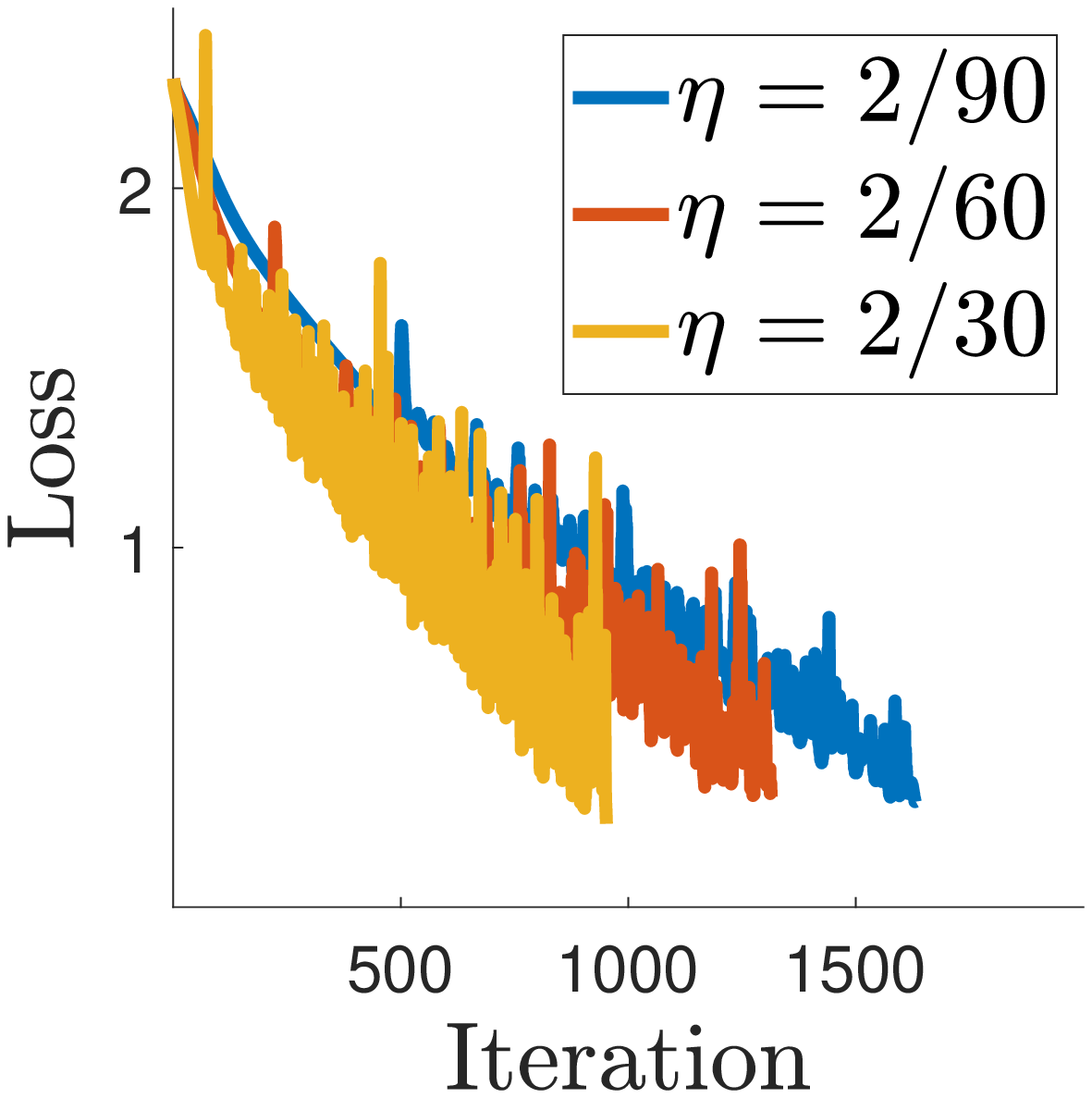} 	\includegraphics[width=0.32\columnwidth]{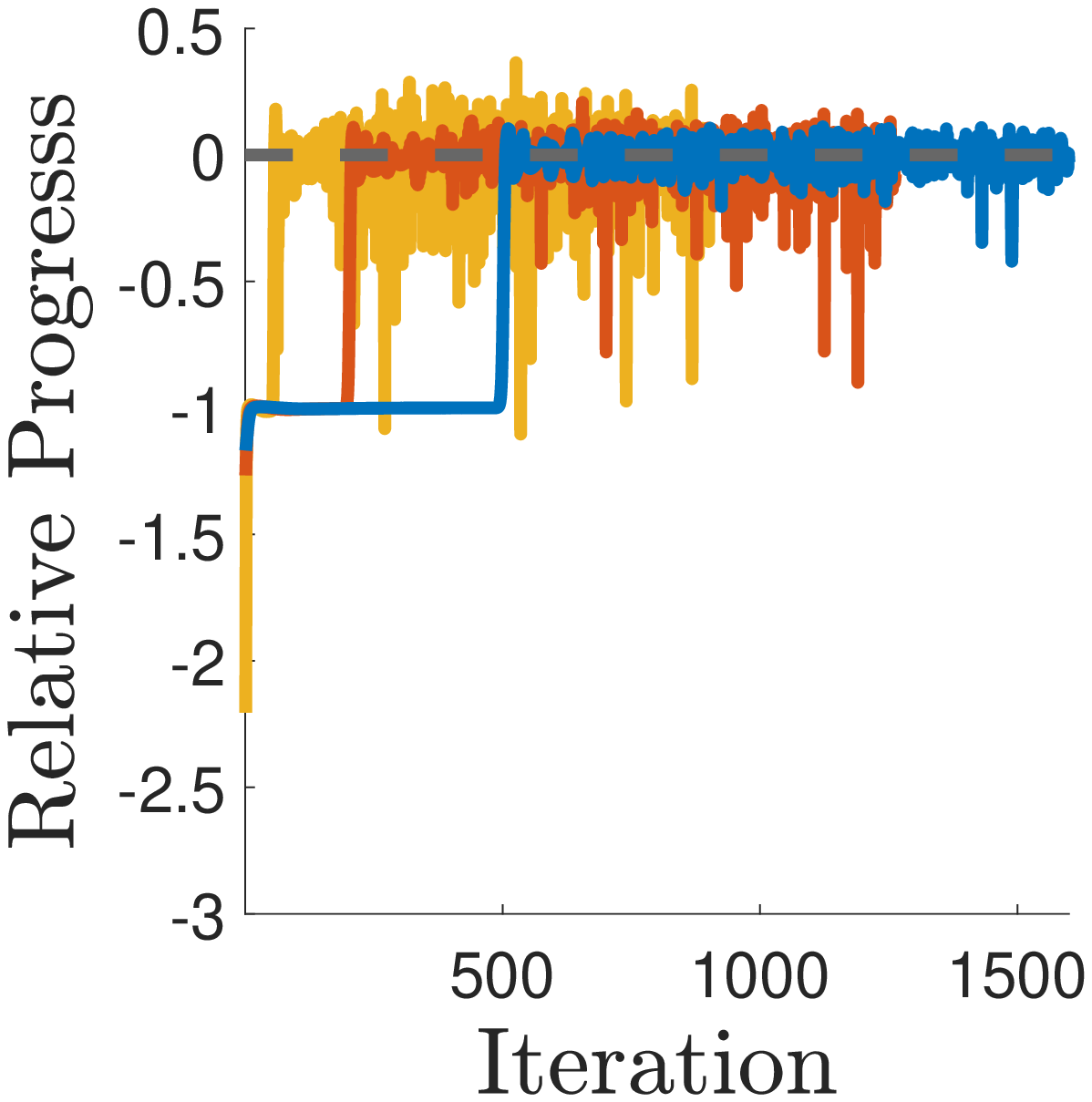}
\includegraphics[width=0.32\columnwidth]{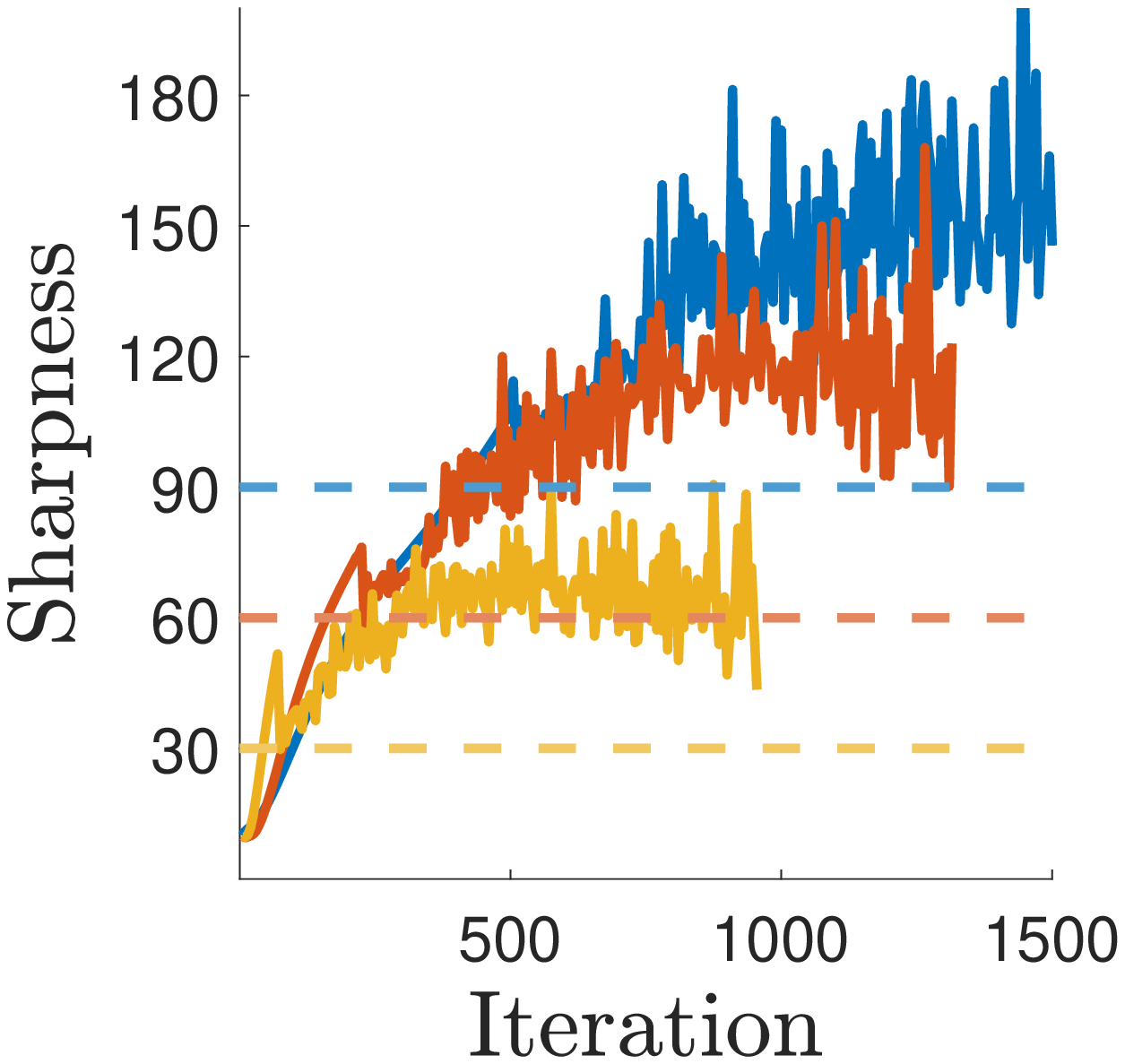}

 }

The above experiment shows that in the \ur, the \rp{} ratio saturates around $0$ unlike the stable regime.
\qqq 
\end{experiment}  
\begin{remark}\label{rmk:fast}
One curious aspect of the above results is that the optimization seems to get faster as we choose larger step sizes.  
This is in fact one of the main observations in \cite{Cohen2021}, suggesting that the unstable convergence is preferred in practice for its faster optimization.
However, that does not mean one can increase the step size too large.
For example, in the above experiment, we observe that the training loss diverges for step size $\eta = 2/10$. 
\end{remark}

Based on \autoref{exp:unstable}, we raise the following question:
\begin{center}
    {\bf Q.} why does $\rr{\zz{t}}$ oscillate around $0$ under unstable convergence?
\end{center}

We begin with explaining why  $\rr{\zz{t}}$ cannot stay  above $0$.
Since the loss is converging in a long term, it cannot be that $\rr{\zz{t}}>0$ for many iterations;
otherwise, the loss will keep increasing, contradicting the convergence.

More curious part is the fact that $\rr{\zz{t}}$ cannot stay below zero, which directly contrasts with the stable regime.
To understand this phenomenon, we begin with some intuition.

We have seen that when GD encounters sharp minima, it  oscillates near the minima 
because it cannot stabilize to the minima (due to \autoref{thm:master}). 
In other words, the loss change $f(\zz{t+1})-f(\zz{t})$ would be much smaller compared to $\norm{\eta^2\nabla f(\zz{t})}^2$ the square of the distance that GD travels. Hence, intuitively, one might expect that \rp{} cannot be too negative under the \uc{} regime. 
We would like to formalize this intuition next.

\subsection{Characteristics in Iterates Movement}
\label{sec:iterate}
To that end, let us formally define what it means for GD to oscillate.
More generally, consider the situation where $\zz{}$ is updated by moving along the vector $-\vv$.
Then this update is oscillatory if the directional derivative at the updated parameter $\zz{}-\vv{}$ is nearly negative of that at $\zz{}$, i.e.,
\begin{align*}
    \inp{\vv}{\nabla f(\zz{} - \vv)}  \approx  -\inp{\vv}{\nabla f(\zz{})}\,.
\end{align*}
Inspired by this, we consider the following definition.

\begin{definition}[Directional smoothness] For an update vector $\vv$, we define
\begin{align*}
      \dir{\vv}{\zz{}}:=\frac{1}{\norm{\vv}^2}\inp{ \vv}{  \nabla f(\zz{})- \nabla f\big(\zz{} -  \vv \big)}.
\end{align*}
\end{definition}

Now coming back to the gradient descent where the update vector is $\vv = \eta \nabla f(\zz{})$,  we have
\begin{align*}
    \dir{\eta \nabla f(\zz{})}{\zz{}} = \frac{\inp{ \nabla f(\zz{})}{  \nabla f(\zz{}) - \nabla f\big(\zz{} - \eta \nabla f(\zz{}) \big)}}{\eta \norm{\nabla f(\zz{})}^2}.
\end{align*}
When GD is exhibiting oscillatory behaviour, we would have 
\begin{align*}
    \inp{\nabla f(\zz{})}{\nabla f(\zz{} - \vv)}  \approx  -\inp{\nabla f(\zz{})}{\nabla f(\zz{})}\,,
\end{align*}
in which case, it holds that 
\begin{align} \label{eq:oscil}
    \dir{\eta \nabla f(\zz{})}{\zz{}} \approx \frac{2}{\eta}\quad \text{(when GD  iterates oscillate)}.
\end{align} 

For intuition, let us quickly verify \eqref{eq:oscil} for quadratic costs.

 \begin{example}[Quadratics] \label{ex:dir_quad}
 Consider a quadratic loss function $f(\zz{}) = \zz{\top} P\zz{}$ with  $P \succeq 0$.  
 Then, the GD update reads $\zz{t+1} = ( I-\eta P) \zz{t}$. For an eigenvector/eigenvalue pair  $(\vq,\lambda )$ of $P$, the quantity $\inp{\vq_{\max} }{\zz{t}}$ evolves as 
\begin{align*}
    \inp{\vq_{\max} }{\zz{t}}  &= \vq^\top (I-\eta P)\zz{t-1} = (1-\eta \lambda) \inp{\vq_{\max} }{\zz{t-1}}\\
    &= (1-\eta \lambda)^t \inp{\vq_{\max} }{\zz{0}}\,.
\end{align*}
This implies that if $\lambda <2/\eta $, then  $\vq^\top \zz{t}\to 0$.
Hence,  if  $\eta =  2/\lmax{} (P)$, then after sufficiently large iterations $t$, we have $\zz{t} \approx (-1)^t \inp{\vq_{\max} }{\zz{0}} \vq_{\max}$,
in which case $ \dir{\eta \nabla f(\zz{})}{\zz{}} \approx \frac{2}{\eta}$. \qqq
\end{example}

Given the above view on ``oscillating'' iterates, we now measure directional smoothness under unstable convergence.
 \begin{experiment}[Directional smoothness in stable and unstable regimes] \label{exp:dir}
Under the same setting as Experiments~\ref{exp:stable} and \ref{exp:unstable}, we measure the value $\dir{\eta \nabla f(\zz{t})}{\zz{t}}$ at each iteration.
 
 {\centering

\includegraphics[width=0.48\columnwidth]{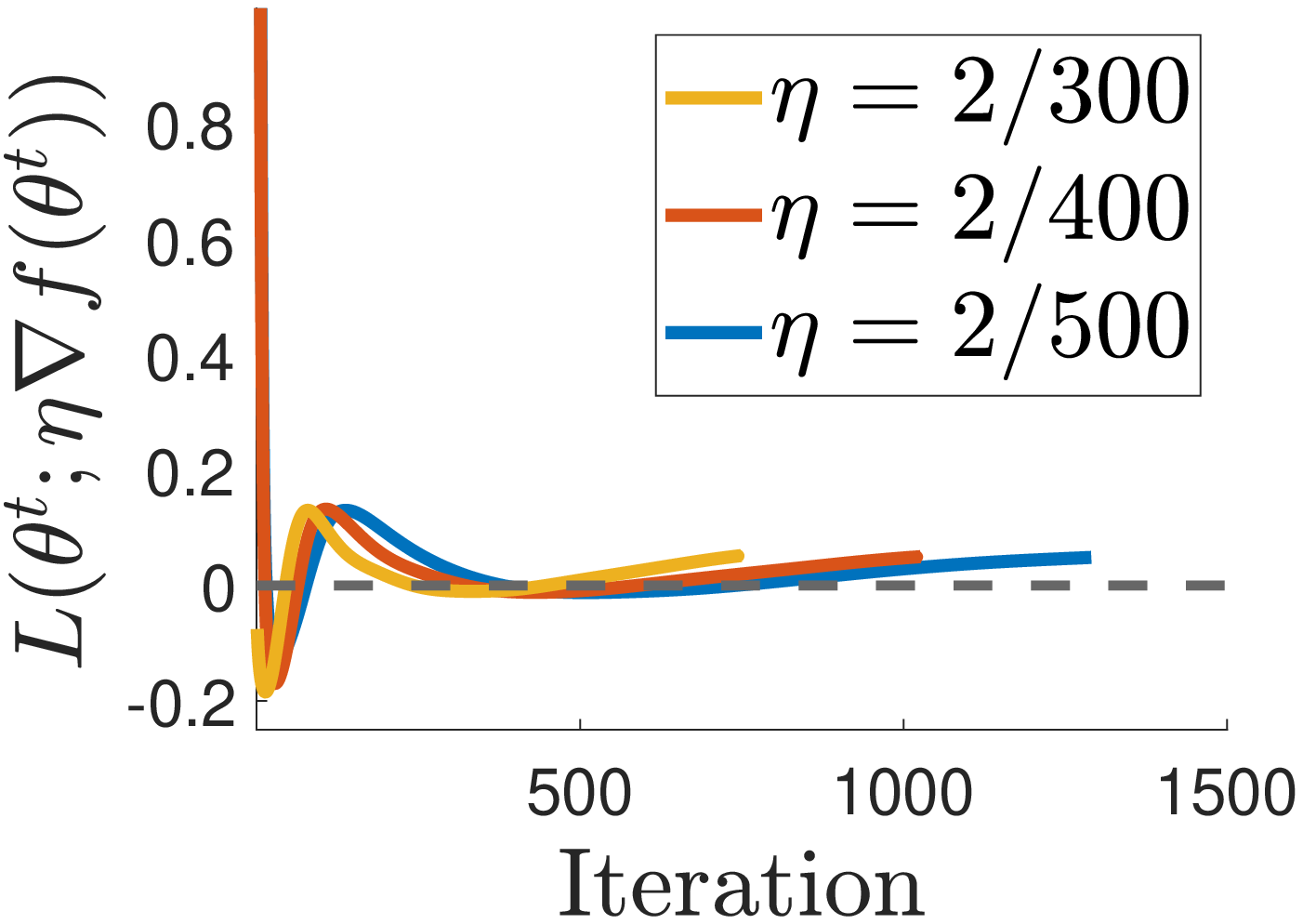}
\includegraphics[width=0.48\columnwidth]{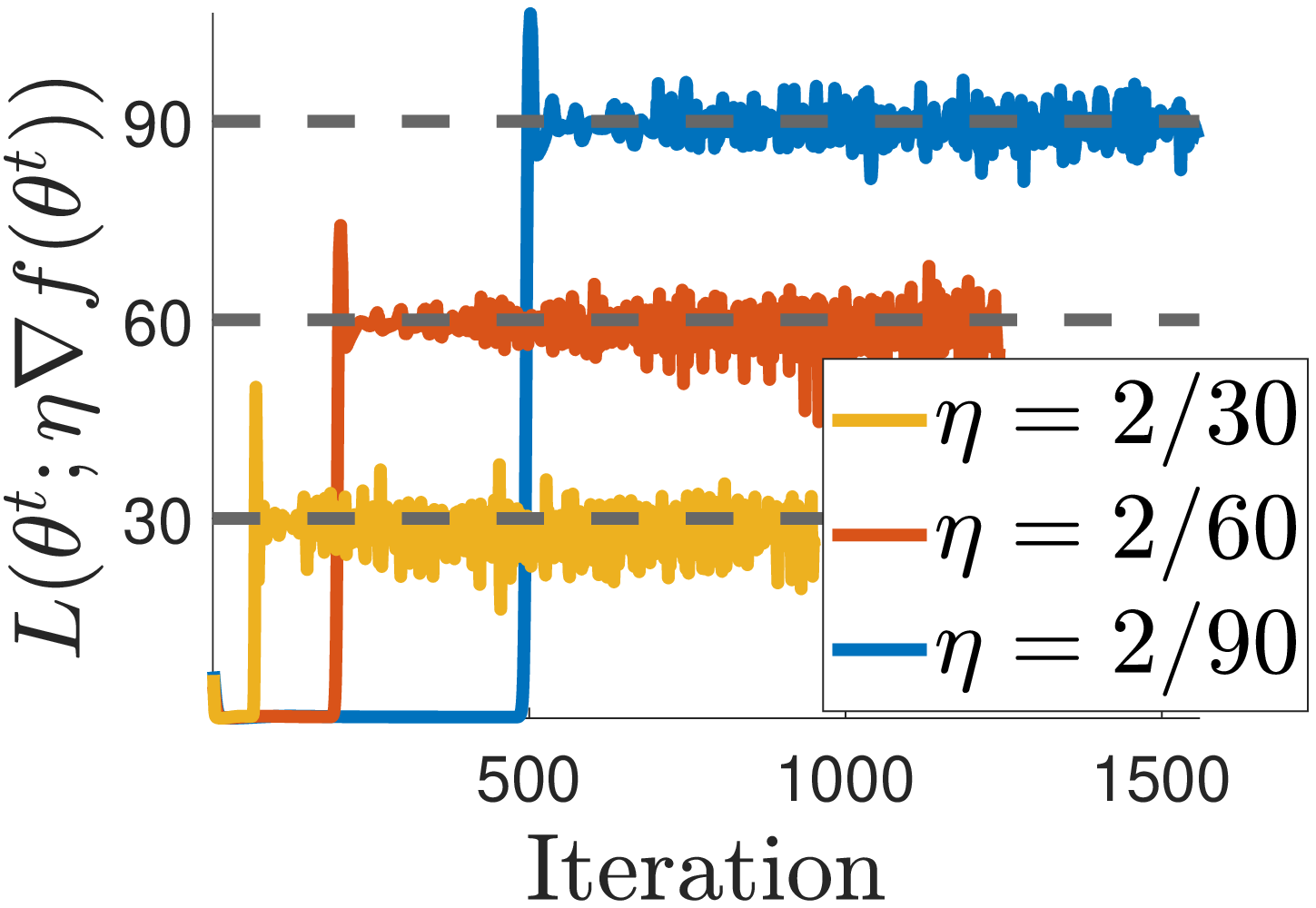}

 }
 
 Indeed, one can see that for the unstable regime $\dir{\eta \nabla f(\zz{t})}{\zz{t}}$ saturates around $2/\eta $, indicating that GD is exhibiting an oscillating behavior.
 \qqq
 \end{experiment}

   \autoref{exp:dir} verifies that GD is indeed showing an oscillating behavior under unstable convergence. We remark that a similar conclusion is made in \cite{xing2018walk} as well as the recent concurrent works \citep{ma2022multiscale,arora2022understanding}. Now coming back to our original question: \emph{can we show a formal relation between the directional smoothness and the relative progress ratio?}
   
   \subsection{Relation between  Relative Progress Ratio and Directionl Smoothness}
   \label{sec:equiv}
  \autoref{thm:equiv} formalizes our intuition that under the oscillating behavior of GD, $\rr{\zz{t}}$ cannot stay below zero.
  
  \begin{theorem} \label{thm:equiv}
  The following identity holds:
   \begin{align}
         \rr{\zz{}} = -1+\frac{\eta}{2}\cdot  2\int_{0}^1 \tau\cdot \dir{\eta \tau \nabla f(\zz{})}{\zz{}} \ \D \tau \,. 
   \end{align}
  \end{theorem}
 \begin{proof}
See \autoref{app:pf:equiv}.
 \end{proof}
 
 \autoref{thm:equiv} implies that if the weighted average of $\dir{\eta \tau \nabla f(\zz{})}{\zz{}}$  is close to $2/\eta$, namely
 \begin{align*}
     2\int_{0}^1 \tau\cdot \dir{\eta \tau \nabla f(\zz{})}{\zz{}} \ \D \tau \approx \frac{2}{\eta}\,,
 \end{align*}
 then $\rr{\zz{}}$ is indeed approximately equal to zero.
 This formally justifies that when  GD  shows an oscillating behavior,  $\rr{\zz{t}}$ cannot stay below zero.

 In our last experiment of this subsection, we verify that the above weighted average is approximately equal to the single value  $\dir{\eta \nabla f(\zz{})}{\zz{}}$, building a stronger relation between the directional smoothness and the relative progress ratio.
 
 \begin{experiment} \label{exp:const}
 In the same setting as \autoref{exp:unstable}, we choose step size $\eta = 2/60$ and in every $5$ iterations, we compute the following values:
 \begin{align*}
     \dir{\eta \tau \nabla f(\zz{t})}{\zz{t}}\quad \text{for}~\tau\in \{0.01,0.02,\dots, 1\}.
 \end{align*}
 In the plot below, we report  the mean of $\dir{\eta \tau \nabla f(\zz{t})}{\zz{t}}$ among $\tau\in \{0.01,0.02,\dots, 1\}$ together with the shades which indicate the standard deviations.

 	 \includegraphics[width=\columnwidth]{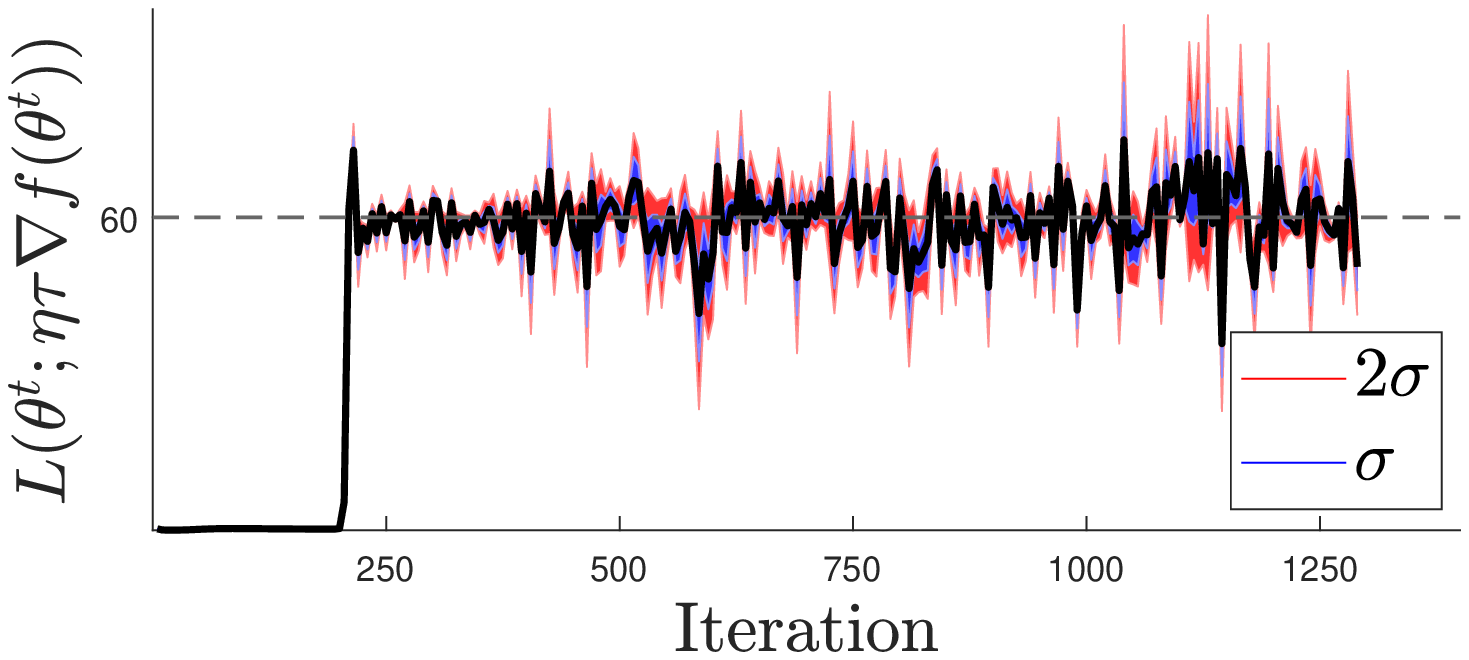}

  This experiment verifies that $\dir{\eta \tau \nabla f(\zz{t})}{\zz{t}}$ does vary too much across $\tau \in [0,1]$. 
  Hence,  the single value  $\dir{\eta \nabla f(\zz{})}{\zz{}}$ well represents the weighted average in \autoref{thm:equiv}. \qqq
 \end{experiment}

 Hence, \autoref{exp:const} justifies the relation
  \begin{align} \label{rel:rp}
        \boxed{ \rr{\zz{}} \approx  -1+\frac{\eta}{2}\cdot    \dir{\eta \nabla f(\zz{})}{\zz{}}\,,} 
 \end{align}
 which precisely explains how the oscillatory behavior of GD results in a small \rp{} ratio.
  
  \begin{remark} \label{rmk:hessian_lip}
  Interestingly, the validity of equation~\eqref{rel:rp} and  \autoref{exp:const} suggests that even though the gradient Lipschitzness is not a good assumption for neural networks, some form of Hessian Lipschitzness is valid along the GD trajectory. 
  \end{remark}
  
  We summarize the finding in this section as follows.
  
  \begin{takeaway}
  Under the unstable convergence regime,  $\rr{\zz{t}}$ oscillates near $0$ for the following two reasons:
         \begin{compactitem}
           \item $\rr{\zz{t}}$ can't stay above $0$ because otherwise the loss would not decrease in the long run.
 
            \item $\rr{\zz{t}}$ can't stay below $0$ due to  the oscillating behavior of GD iterates. This is formalized via \eqref{rel:rp}.
        \end{compactitem}
  \end{takeaway}

\begin{remark}\label{rmk:stable}
Given \eqref{rel:rp}, one can have a better explanation for \autoref{rmk:stable-1} regarding why $\rr{\zz{t}}$ saturates around $-1$.
In the second result of \autoref{exp:dir}, the directional smoothness remains very small in the stable regime. 
Based on \eqref{rel:rp}, this implies that  $\rr{\zz{t}}$ is close to $-1$, which was indeed the case in \autoref{exp:stable}.
\end{remark}

 \subsubsection{Additional experiments} 
In this subsection, we verify the relation \eqref{rel:rp} for other experimental settings.

\begin{experiment}[CIFAR-10; ReLU networks] \label{exp:relu}
Under the same setting as \autoref{exp:stable} (the setting of the main experiments in \cite{Cohen2021}), this time we choose ReLU as activation functions. 

{\centering 

\includegraphics[width=0.32\columnwidth]{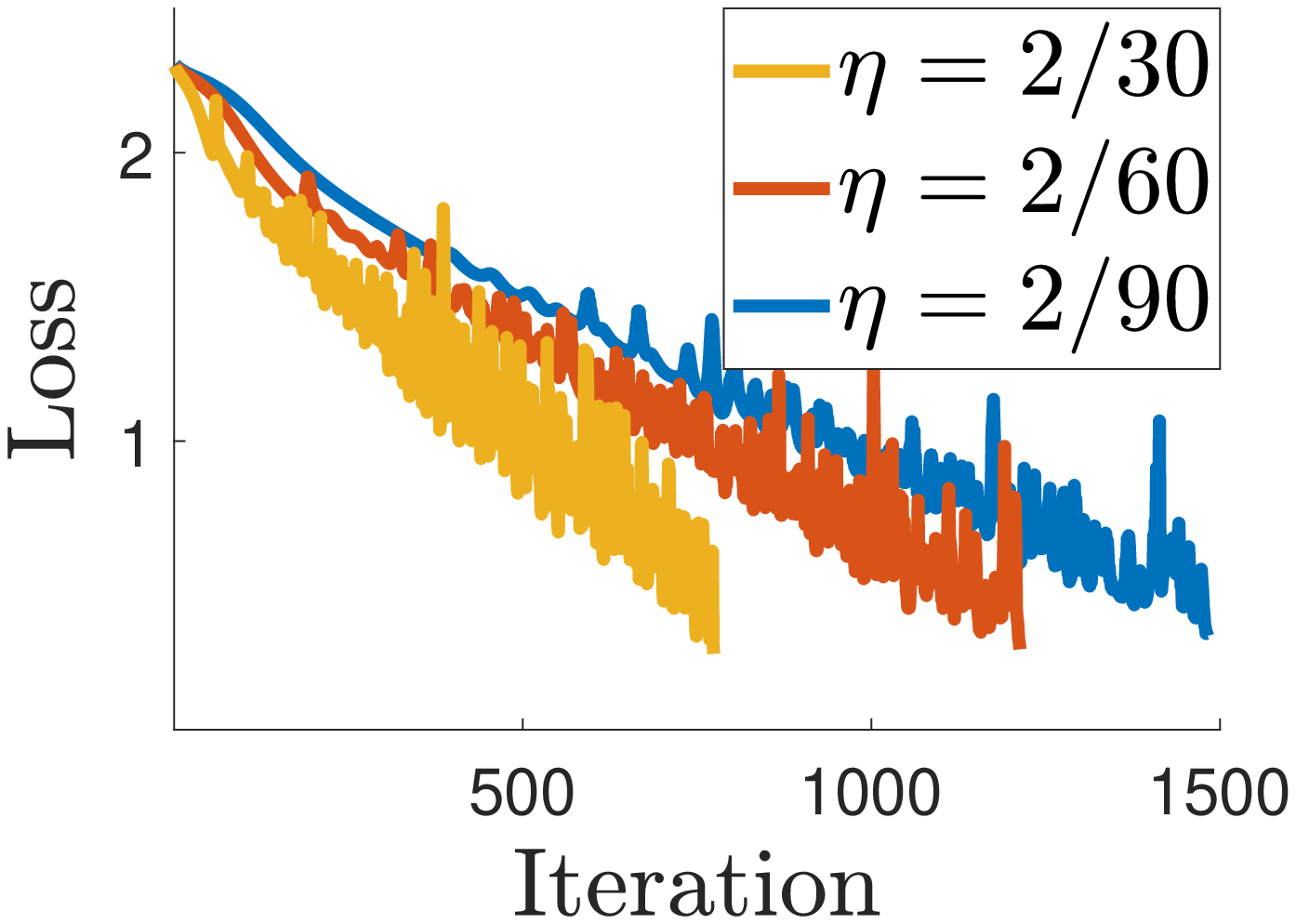} 
\includegraphics[width=0.32\columnwidth]{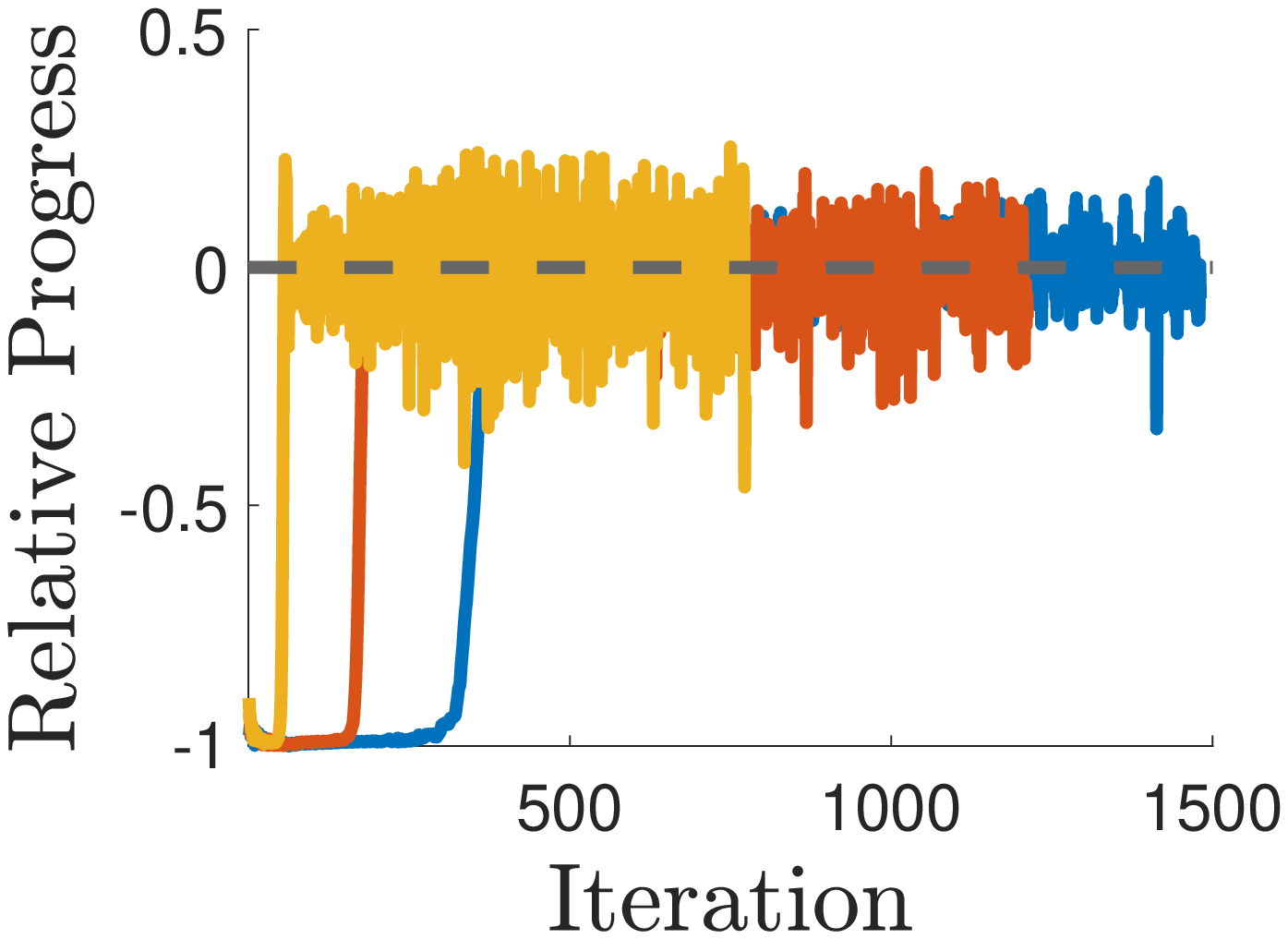}	\includegraphics[width=0.32\columnwidth]{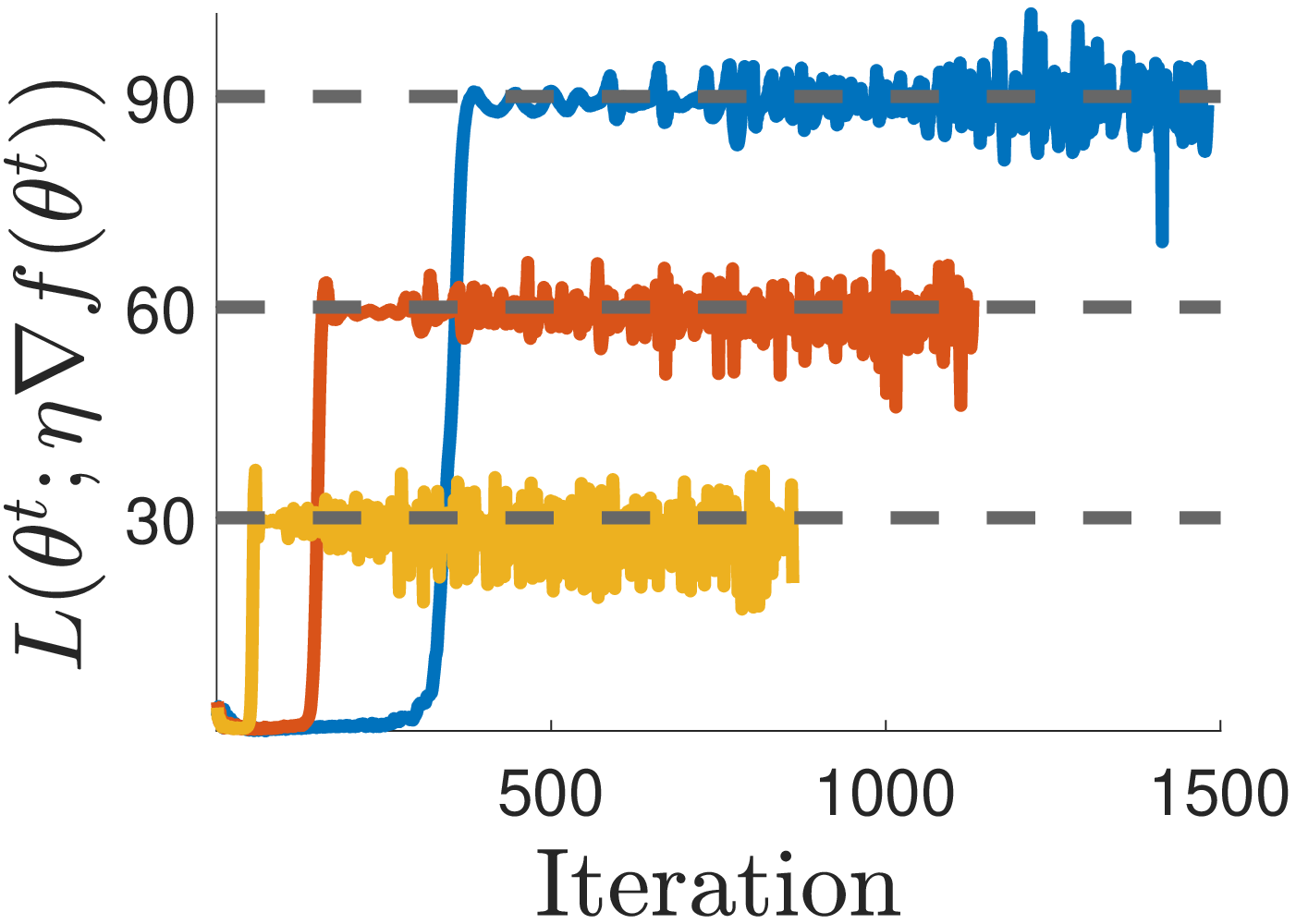}

}

In the next set of experiments, we put $2$ more hidden layers of width $200$ (total $4$ hidden layers of width $200$).

 \includegraphics[width=0.32\columnwidth]{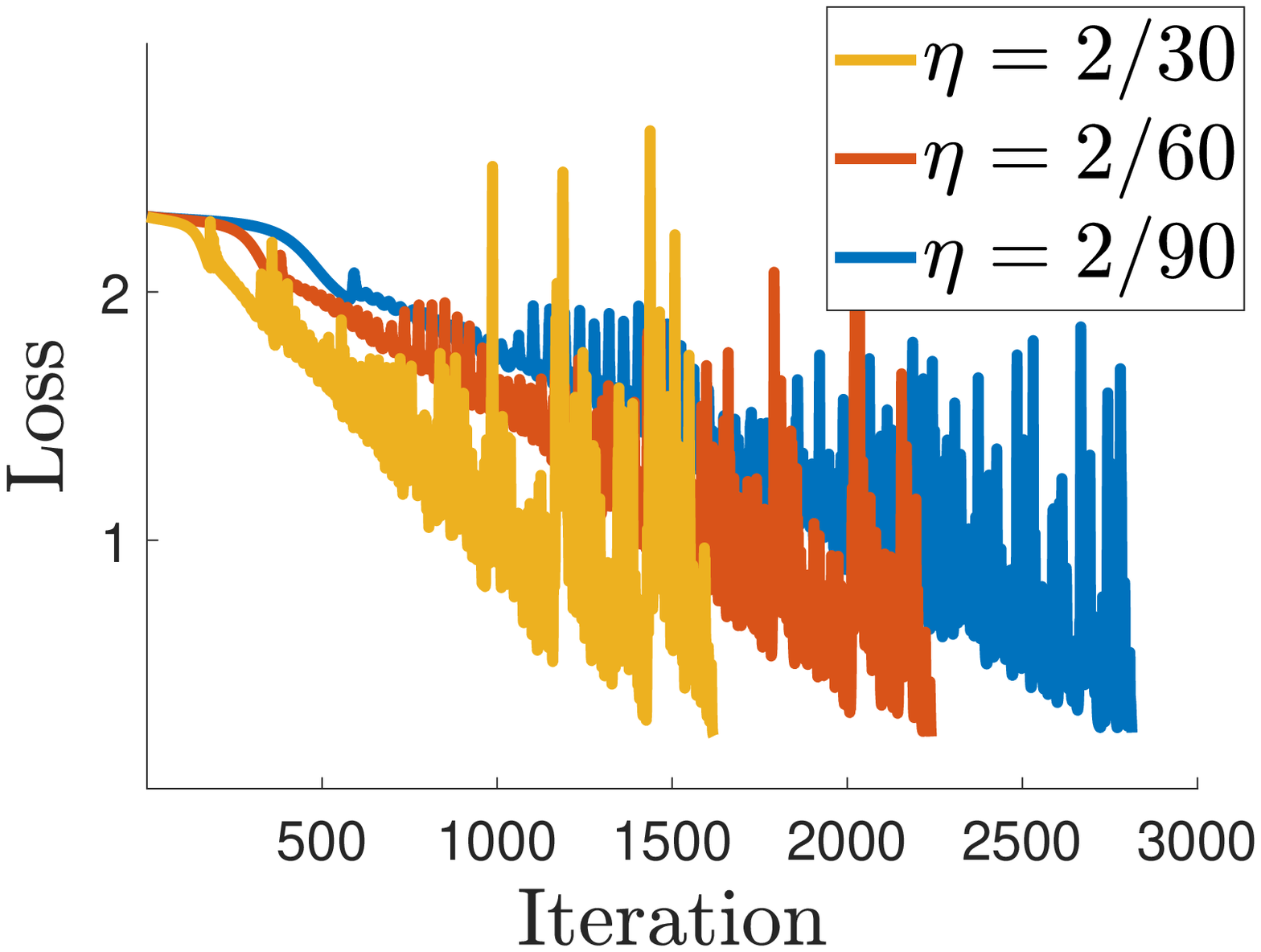} \includegraphics[width=0.32\columnwidth]{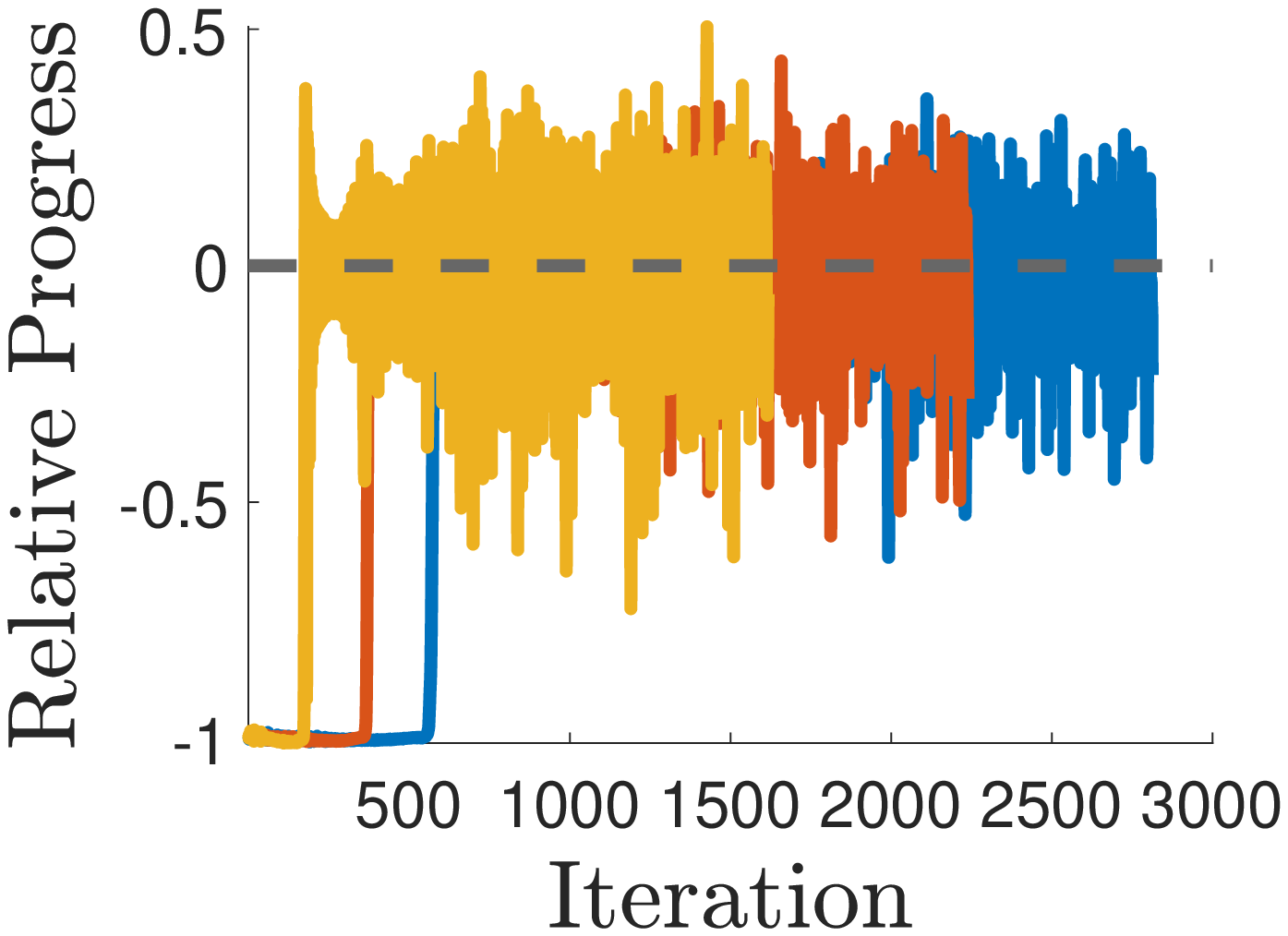}\includegraphics[width=0.32\columnwidth]{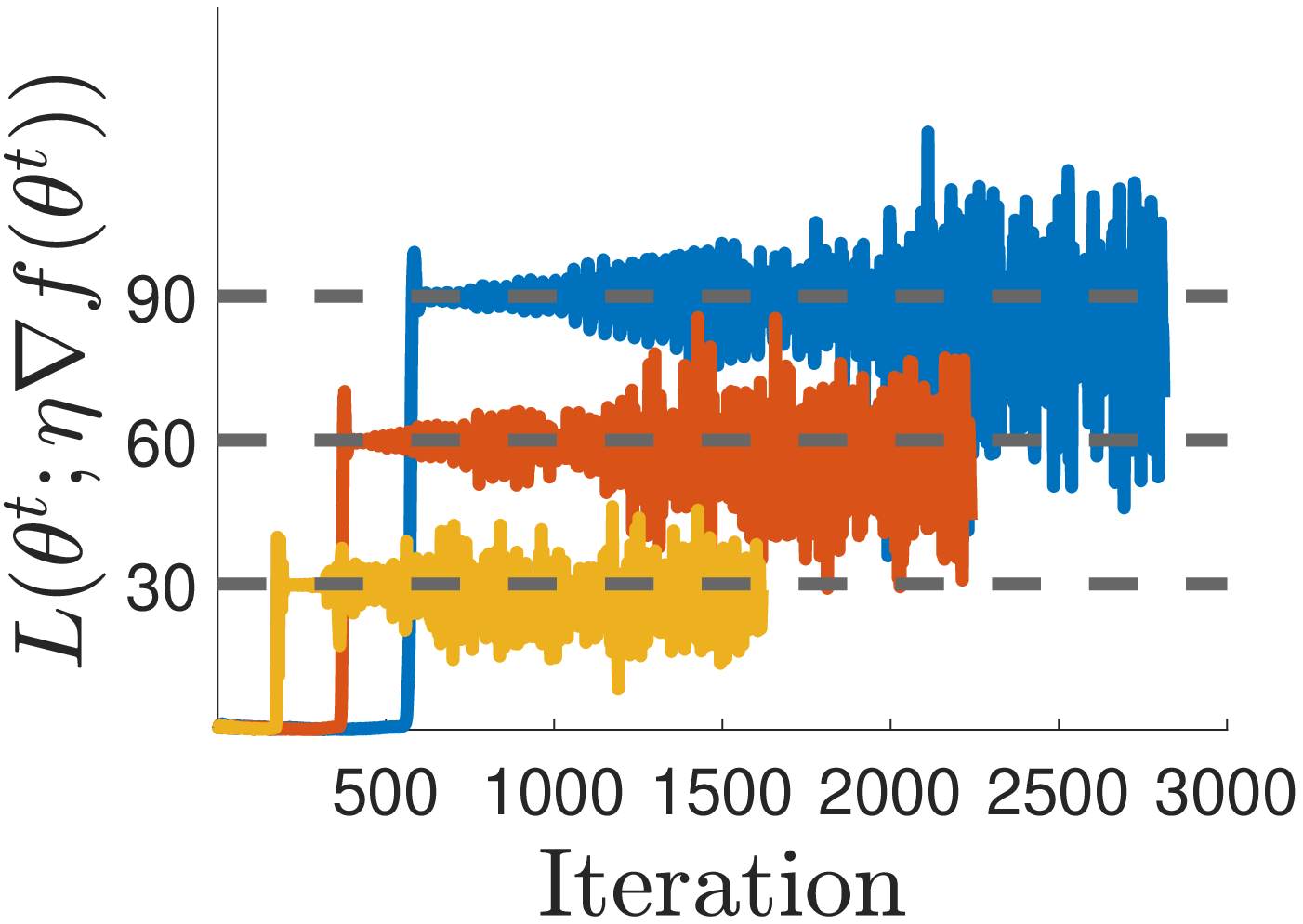}

 The results are largely similar to those for $\tanh$ activations, and the relation \eqref{rel:rp} holds for all cases.
 \qqq
\end{experiment}

\subsection{Implications for Sharpness}

\label{sec:sharp}
   
   In the previous subsection, we saw that one characteristics of unstable convergence is that  $\rr{\zz{t}}$ saturates around zero.
  In this section, we investigate implications of this characteristics for sharpness.  
  In particular, we discuss some relations to a curious phenomenon called \emph{edge of stability} (EoS) recently observed in \cite{Cohen2021}.
  The gist of their observation is that for GD on neural networks often satisfies the following properties:  
  \begin{compactitem}
    \item[A.]  $\lmax{}(\nabla^2 f(\zz{t}))> 2/\eta$ for most iterates.
    \item[B.] In fact, in many cases  $\lmax{}(\nabla^2 f(\zz{t}))$ saturates right at (or slightly above)  $2/\eta$.
  \end{compactitem}

  To that end, we begin with the following consequence of \autoref{thm:equiv}.
  \begin{corollary}\label{cor:sharpness}
Let $\sm{t}$ be the maximum sharpness along the line segment between the iterates $\zz{t}$ and $\zz{t+1}$, i.e., $\sm{t}:=\sup_{\zz{} \in \overline{\zz{t}\zz{t+1}}}\{\lmax{}(\nabla^2 f(\zz{})) \}$.  Then, the following inequality holds:
  \begin{align*}
   \frac{2}{\eta }\cdot (\rr{\zz{}} +1) \leq    \sm{t} \,.   
  \end{align*} 
  \end{corollary}
  \begin{proof}
  It follows from the fact that for each $\tau \in [0,1]$ $\dir{\eta \tau \nabla f(\zz{})}{\zz{}} \leq \sup\{ \lmax{}(\nabla^2 f(\zz{}))~:~\zz{}$ lies on the line segment between $\zz{t}$ and $\zz{t}-\eta\tau \nabla f(\zz{})\}$. Clearly, the right hand side is upper bounded by $\sm{t}$. 
  \end{proof}

\autoref{cor:sharpness} implies that when $\rr{\zz{t}}$ oscillates around zero, then $\sm{t}$ has to be frequently above the threshold $2/\eta$.
One can actually refine this statement to understand the part A of EoS, given our results so far.
In light of \autoref{exp:const}, if 
$\dir{\eta \tau \nabla f(\zz{t})}{\zz{t}}$ does vary much across $\tau \in [0,1]$, then one can actually write
\begin{align*}
   \frac{2}{\eta }  (\rr{\zz{t}} +1) &\approx \lim_{\tau \to 0 } \dir{\eta \tau \nabla f(\zz{t})}{\zz{t}}\\ &\overset{\clubsuit}{\leq} \lmax{}(\nabla^2 f(\zz{t})) \,,     \end{align*}
  which is the part A of EoS.
  
  Moreover, let us for a moment additionally assume that $\nabla f(\zz{t})$ is approximately parallel to the largest eigenvector of the Hessian $\nabla^2 f(\zz{t})$.
  This might look stringent at first glance, but given the calculations in \autoref{ex:dir_quad}, this assumption is true for unstable GD on a quadratic cost.
  Also, recently, this behavior is theoretically proven for the normalized gradient descent dynamics~\cite{arora2022understanding}.
  Under this assumption,  one can further deduce that the inequality $(\clubsuit)$ holds with approximate equality, and the part B of EoS would hold in that case.

\section{Relative Progress for SGD}
\label{sec:sgd}

In this section, we extend our discussion to the stochastic gradient descent (SGD):
 \begin{align*} 
     \zz{t+1} = \zz{t} - \eta g(\zz{t}),~~\text{where }\E[g(\zz{t})] = \nabla f(\zz{t})\,.
 \end{align*}
For the case of SGD, there is one obvious challenge.
With SGD, the training loss does not decrease monotonically  since SGD is a random algorithm.
Hence, it is not clear how to precisely define what it means for SGD to be in the unstable regime.
On the other hand, inspired by our discussion in \autoref{sec:loss}, a  more transparent way to define the unstable regime for SGD is via the \rp{} ratio.
In particular, we consider the following extension.
\begin{definition}[Expected relative progress ratio]
\begin{align*}
    \E[\rr{\zz{}}]:=\frac{\E f(\zz{}-\eta g(\zz{}) )-f(\zz{})}{\eta \norm{\nabla f(\zz{})}^2}\,.
\end{align*}
\end{definition}

\begin{experiment}[CIFAR-10; SGD on ReLU networks] \label{exp:sgd_relu}
Under the same setting as \autoref{exp:relu}, this time we train the network with SGD with minibatch size of $32$ and step size $\eta=2/100$. We compute the full-batch loss and the expected \rp{} ratio at the end of each epoch.

{\centering	
		\includegraphics[width=0.32\columnwidth]{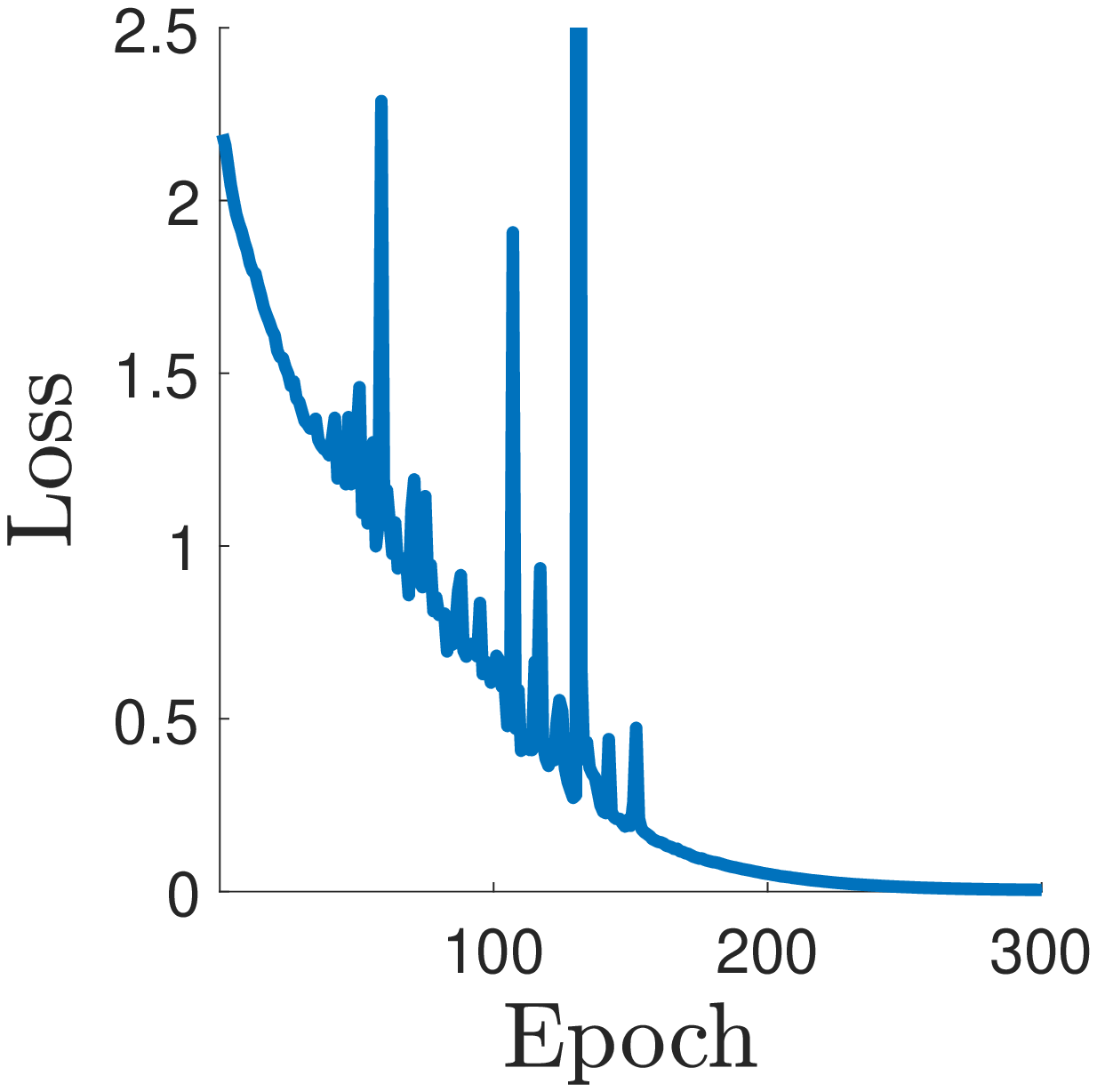}	\includegraphics[width=0.32\columnwidth]{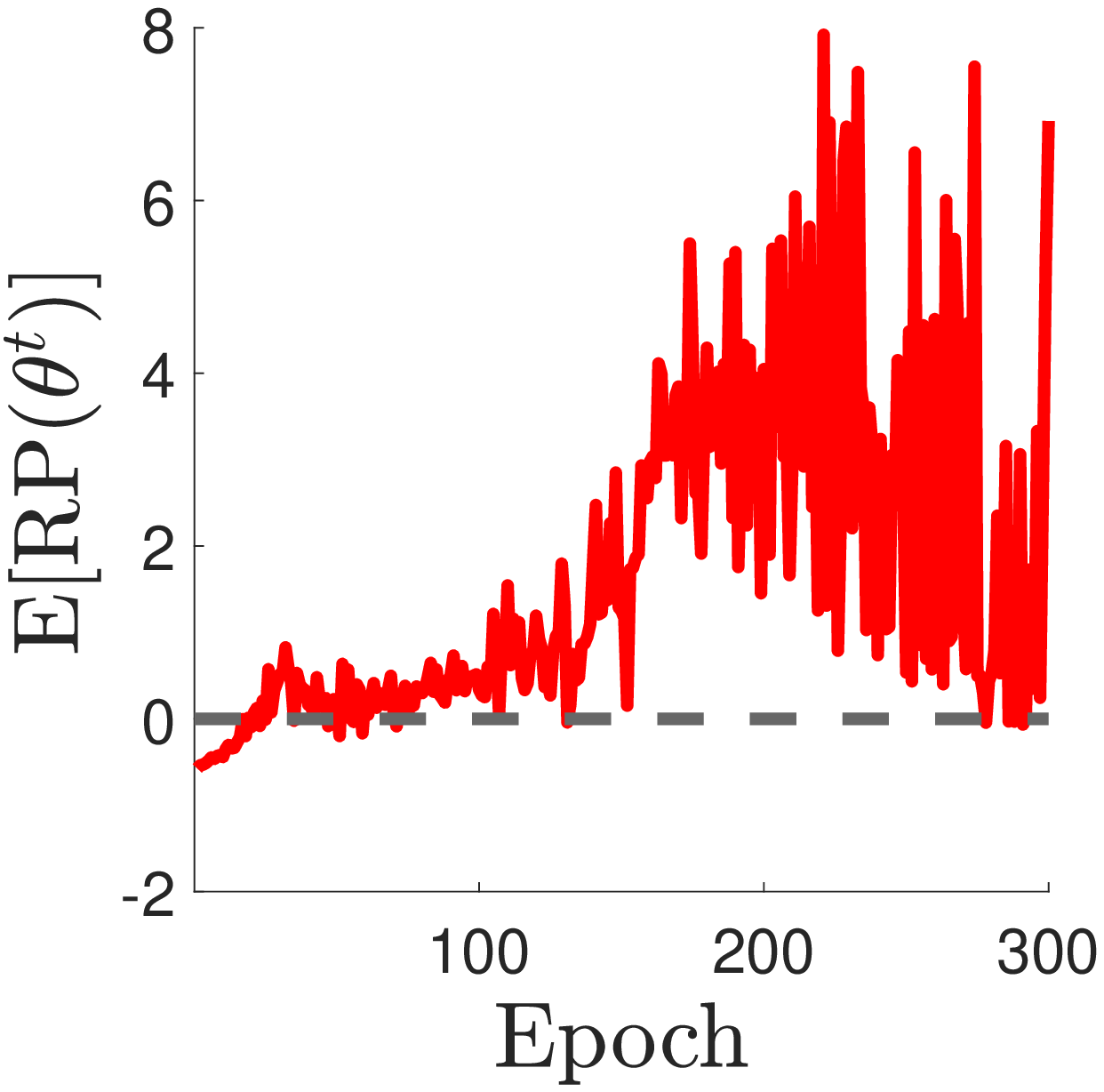}
		\includegraphics[width=0.32\columnwidth]{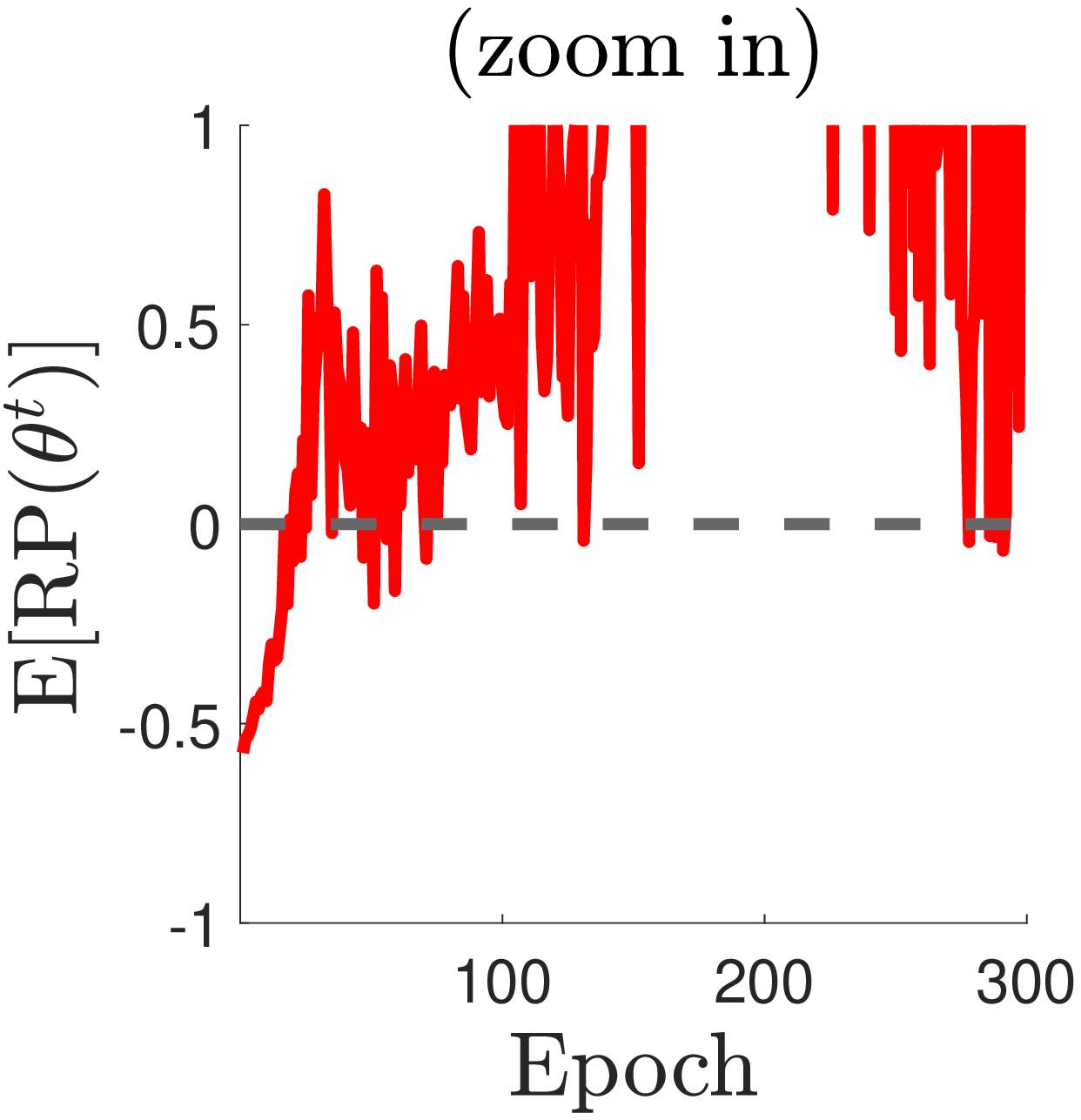}

	}
 
Note that $\E[\rr{\zz{t}}]$ does not stay below zero. Based on our discussion in \autoref{sec:loss}, this suggests that SGD is in the unstable regime.
\qqq
\end{experiment}
\begin{remark}[{Expected loss change is not negative?!}]
One very surprising aspect of the above results is that  $\E[\rr{\zz{t}}]$ is not negative for a majority of iterations. 
This is rather counter-intuitive given that in the loss plot SGD decreases the loss in the long run.
On the other hand, we note that this counter-intuitive phenomenon is also observed by \cite{Cohen2021} in a comprehensive set of experiments.
In particular, they mention ``\emph{what may be more surprising
is that SGD is not even decreasing the training loss in expectation}.''
See \citep[Appendix H]{Cohen2021}  and \citep[Figures 25 and 26]{Cohen2021} for details.
\end{remark}

We now establish a relation analogous to \eqref{rel:rp} for SGD. Similarly to \autoref{thm:equiv},  one can prove the following:
\begin{align*}
  &\E[\rr{\zz{}}]  =-1 +  \frac{\eta}{2} \cdot 2 \int_{0}^1 \tau  \E_g\left[{\textstyle \frac{\norm{g(\zz{})}^2 \cdot  \dir{\eta \tau g(\zz{})}{\zz{}}}{\norm{\nabla f(\zz{})}^2}  }\right]  \  \D \tau   \,.
\end{align*}
See \autoref{app:pf:equiv} for derivation.
Hence, the analogous relation to \eqref{rel:rp} is:
  \begin{align} \label{rel:sgd} 
      \E[ \rr{\zz{}}] \approx  -1+\frac{\eta}{2}\cdot   \E\left[\textstyle \frac{\norm{g(\zz{})}^2}{\norm{\nabla f(\zz{})}^2} \dir{\eta  g(\zz{})}{\zz{}}\right]. 
 \end{align}

\begin{experiment}[CIFAR-10; verifying \eqref{rel:sgd} for ReLU] \label{exp:sgd_verify} Under the same setting as \autoref{exp:sgd_relu}, we compute $\E[ \rr{\zz{}}]$ and the RHS of \eqref{rel:sgd} at the end of each epoch and compare those values.
We choose step sizes $\eta = 2/50,\ 2/100,\ 2/150$.

{\centering

\includegraphics[width=0.32\columnwidth]{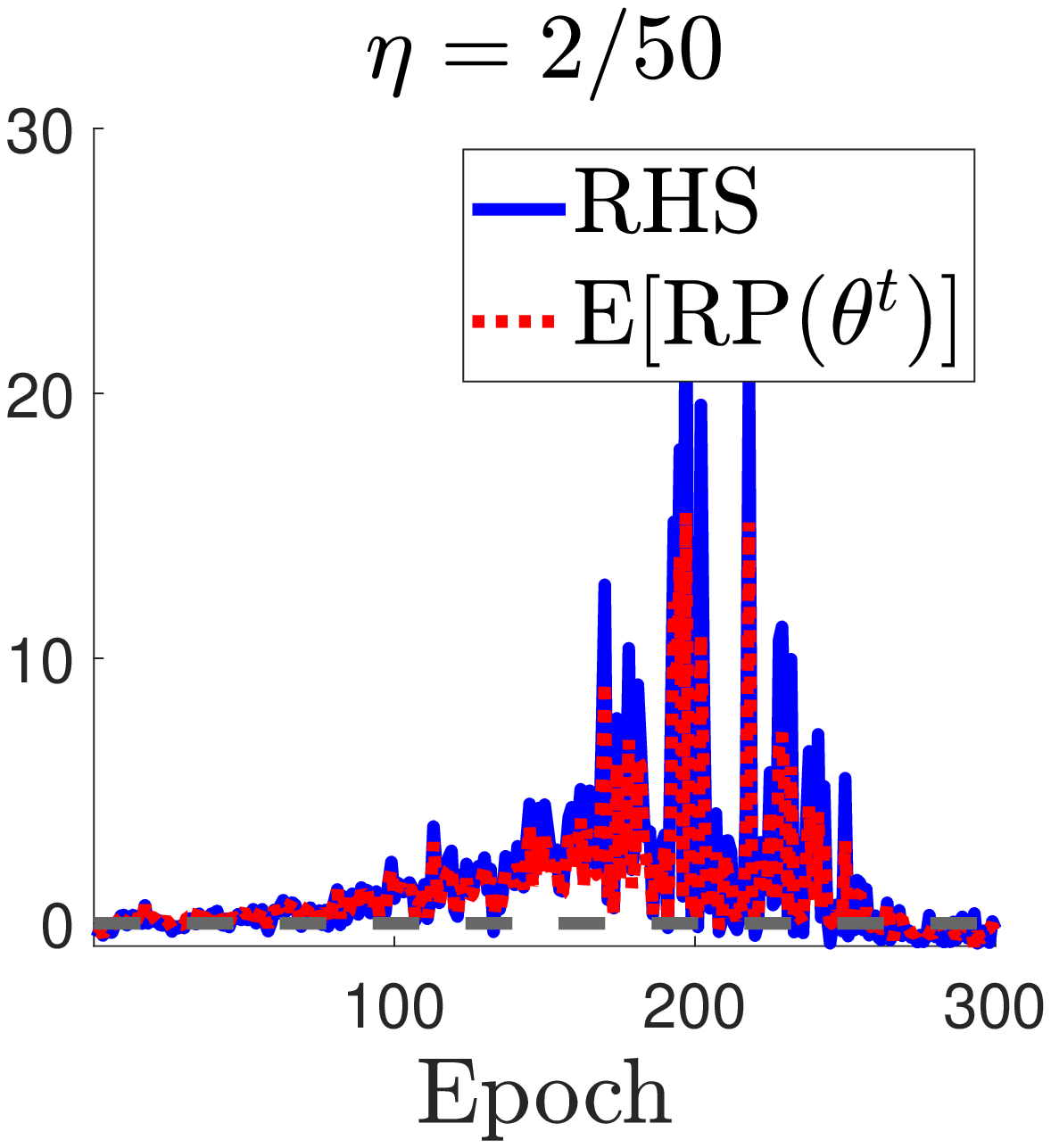} 
\includegraphics[width=0.32\columnwidth]{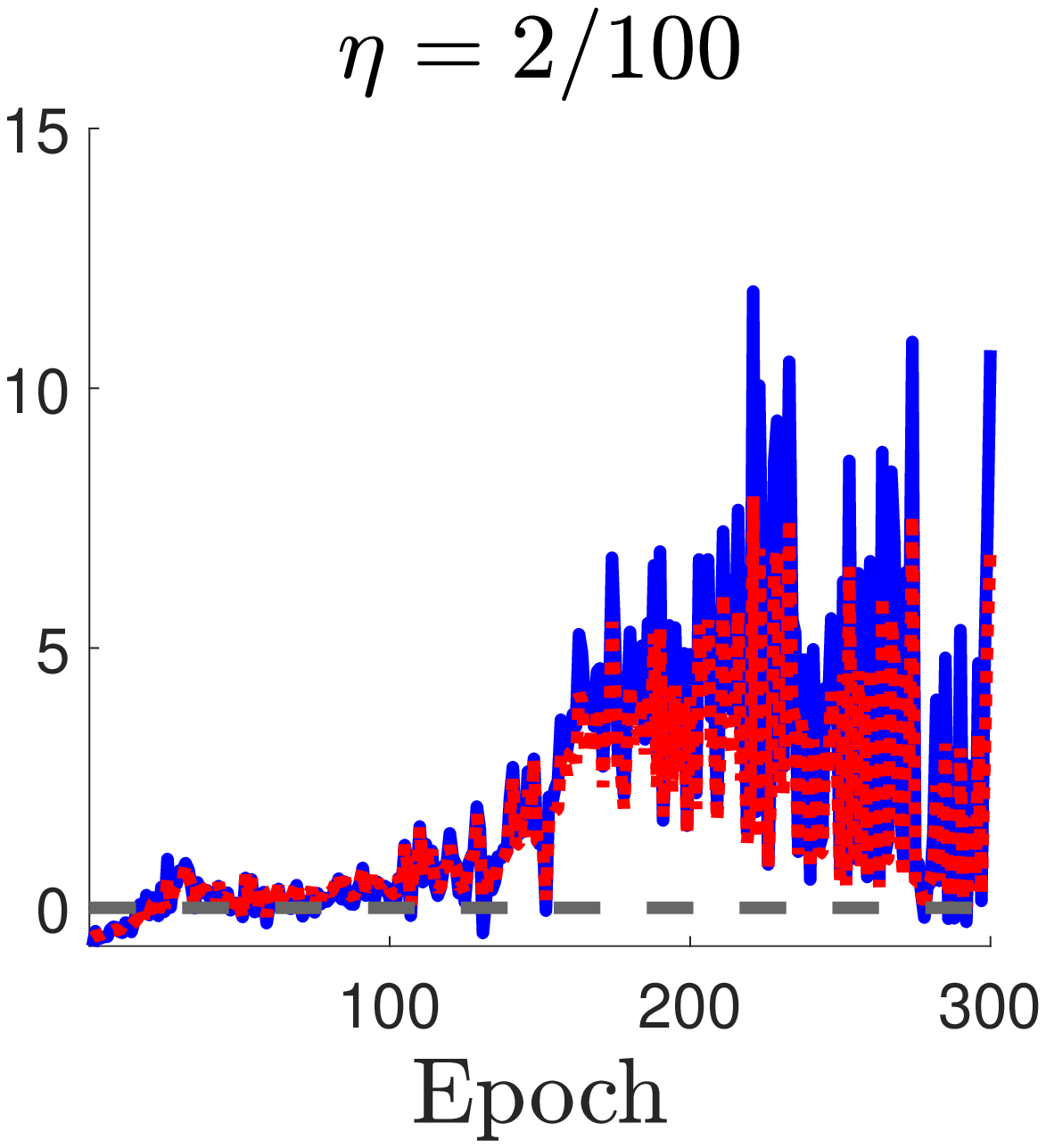} 
\includegraphics[width=0.32\columnwidth]{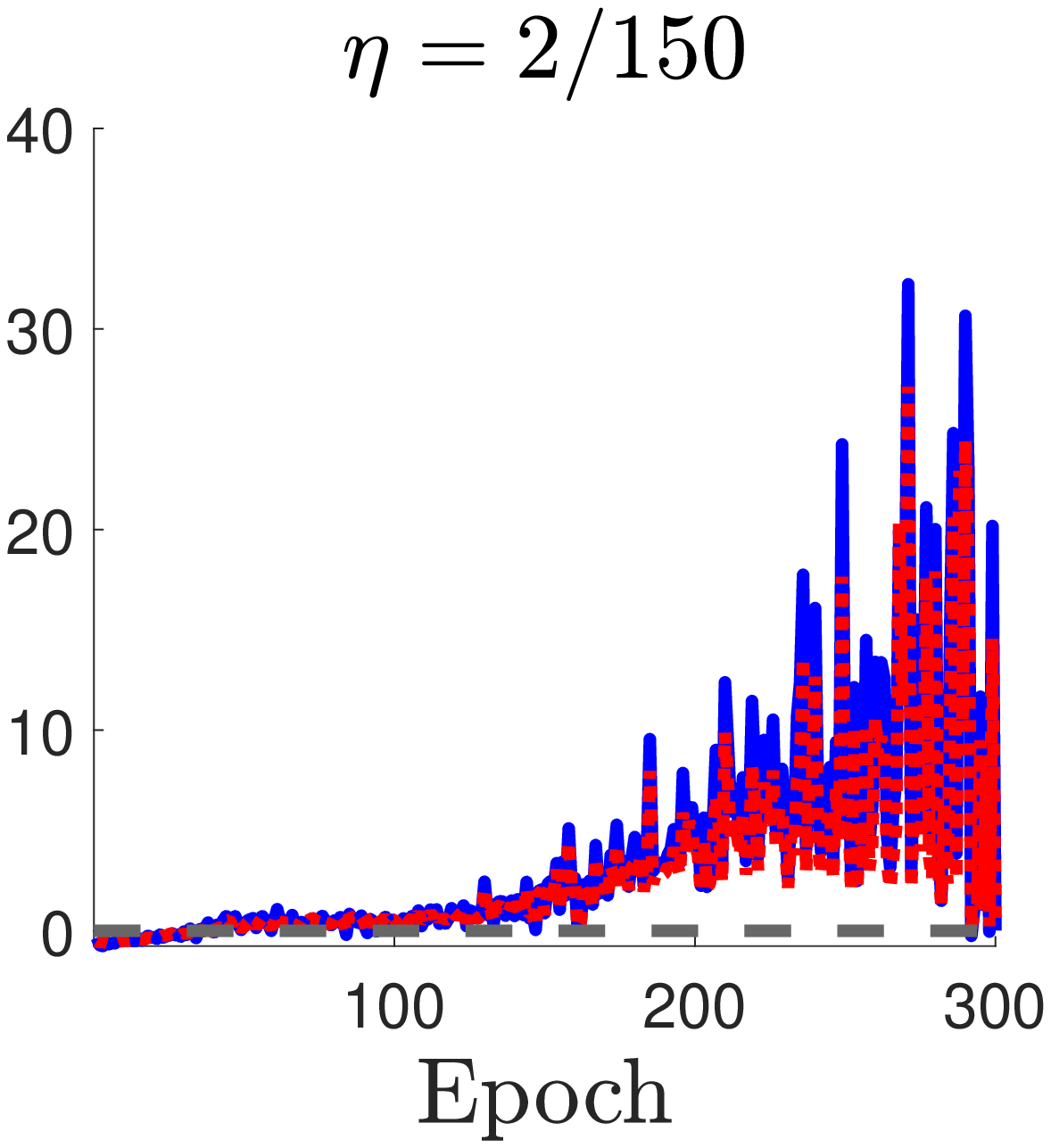}

 }
 Note that the LHS and RHS of \eqref{rel:sgd} are very similar in all results, verifying  the relation \eqref{rel:sgd}.
 \qqq
\end{experiment}

\begin{experiment}[CIFAR-10; verifying \eqref{rel:sgd} for $\tanh$] \label{exp:sgd_verify2} 

 We repeat \autoref{exp:sgd_verify} with $\tanh$ activations.

{\centering

\includegraphics[width=0.32\columnwidth]{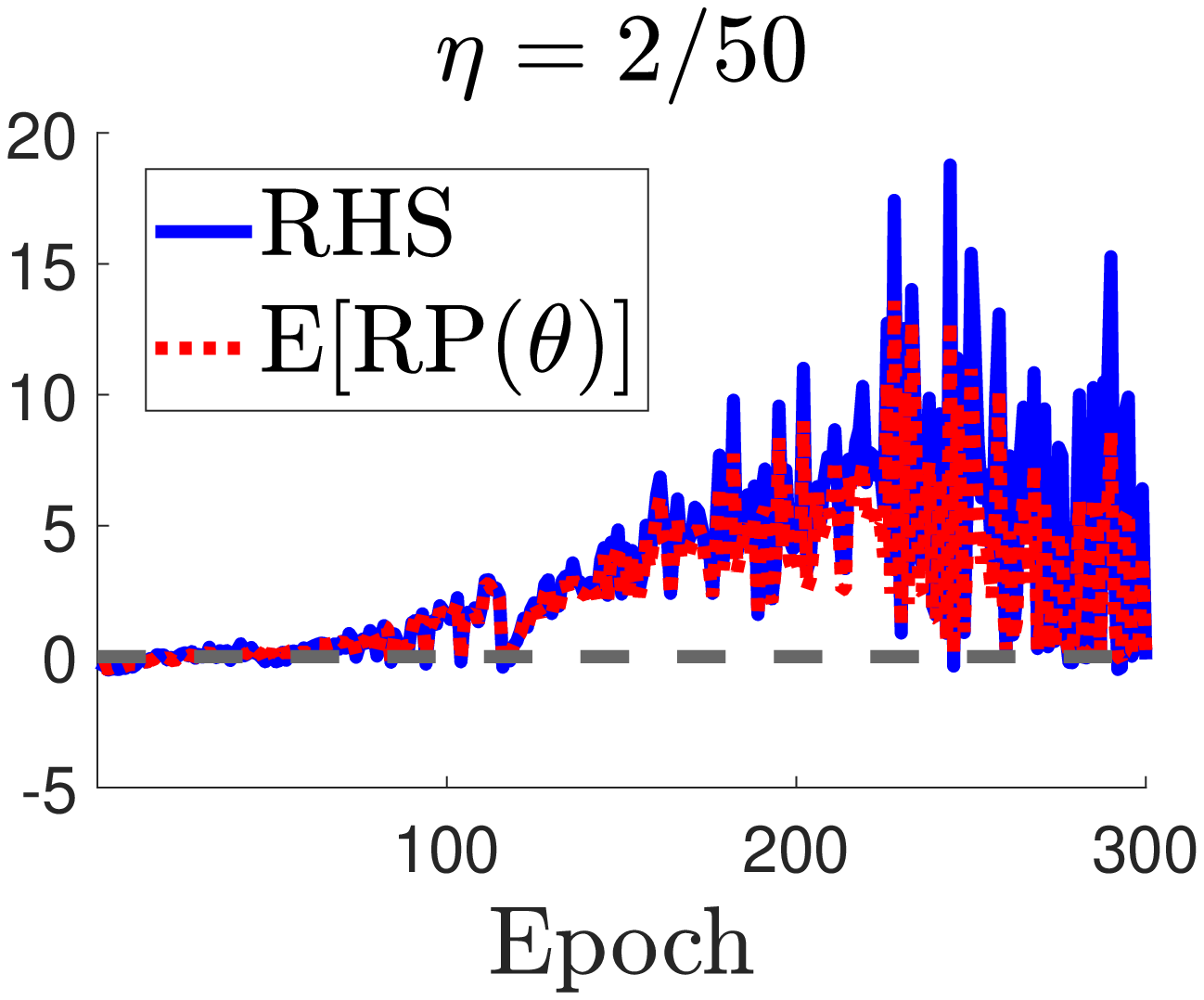}  
\includegraphics[width=0.32\columnwidth]{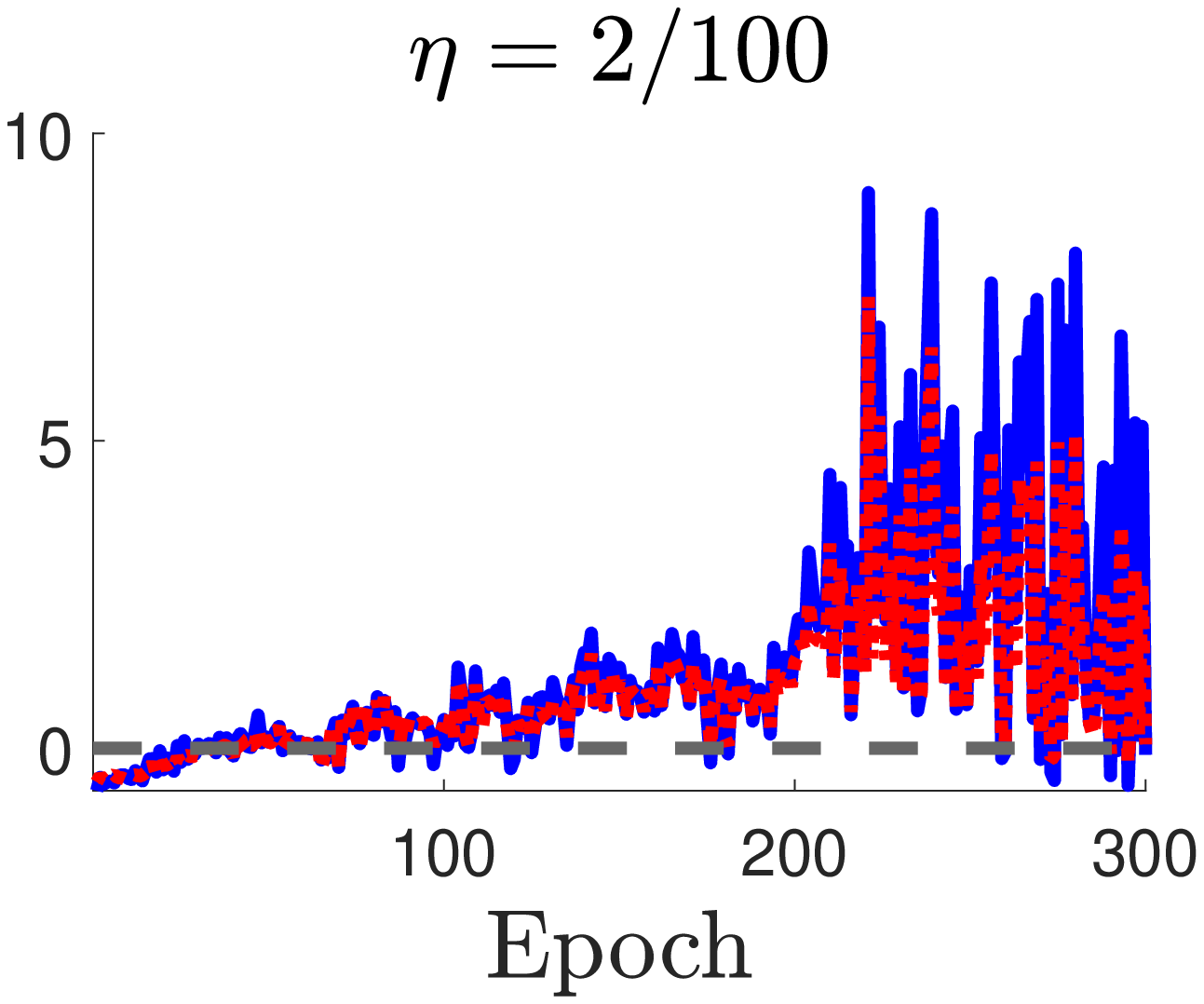}	\includegraphics[width=0.32\columnwidth]{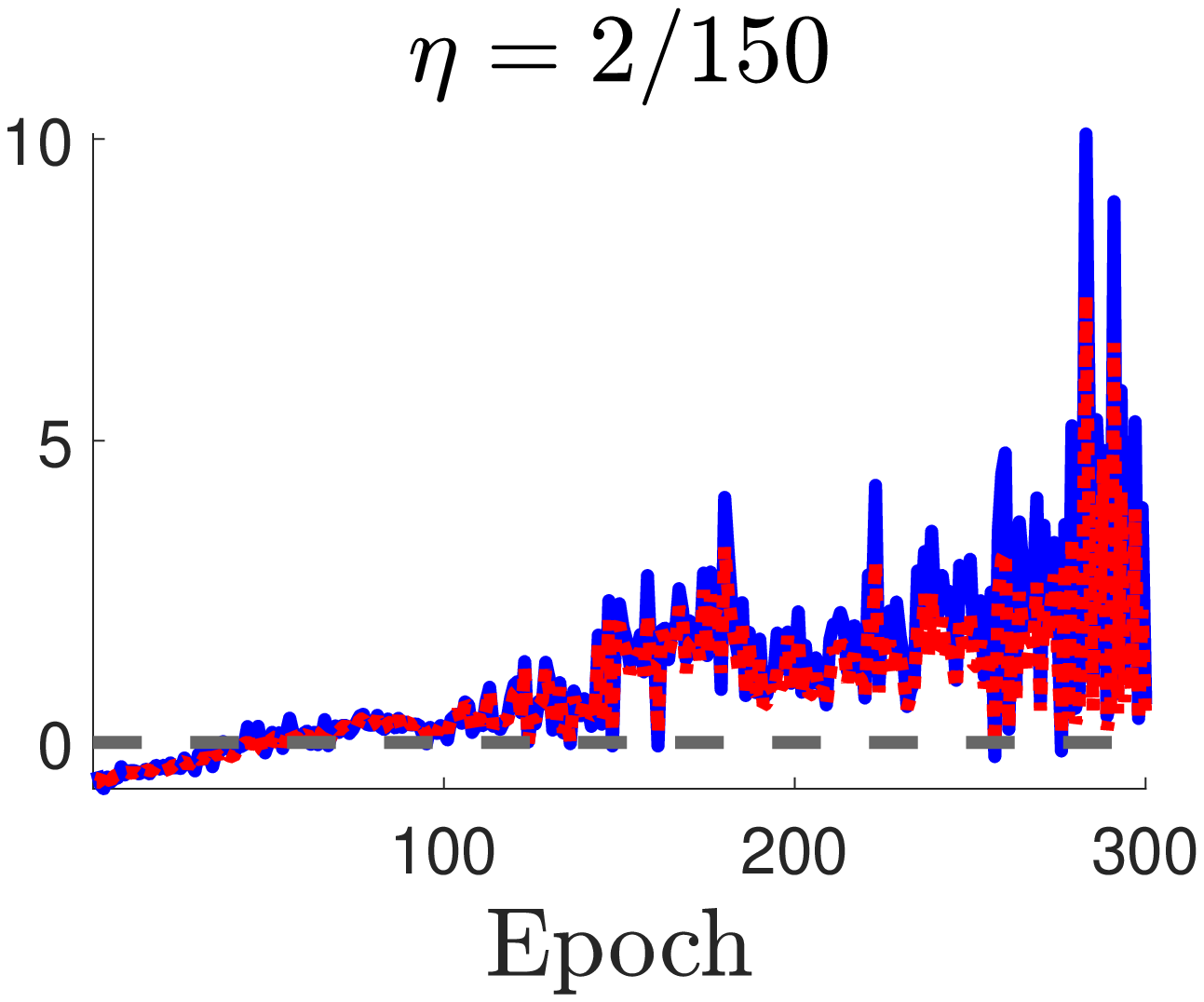}

\includegraphics[width=0.32\columnwidth]{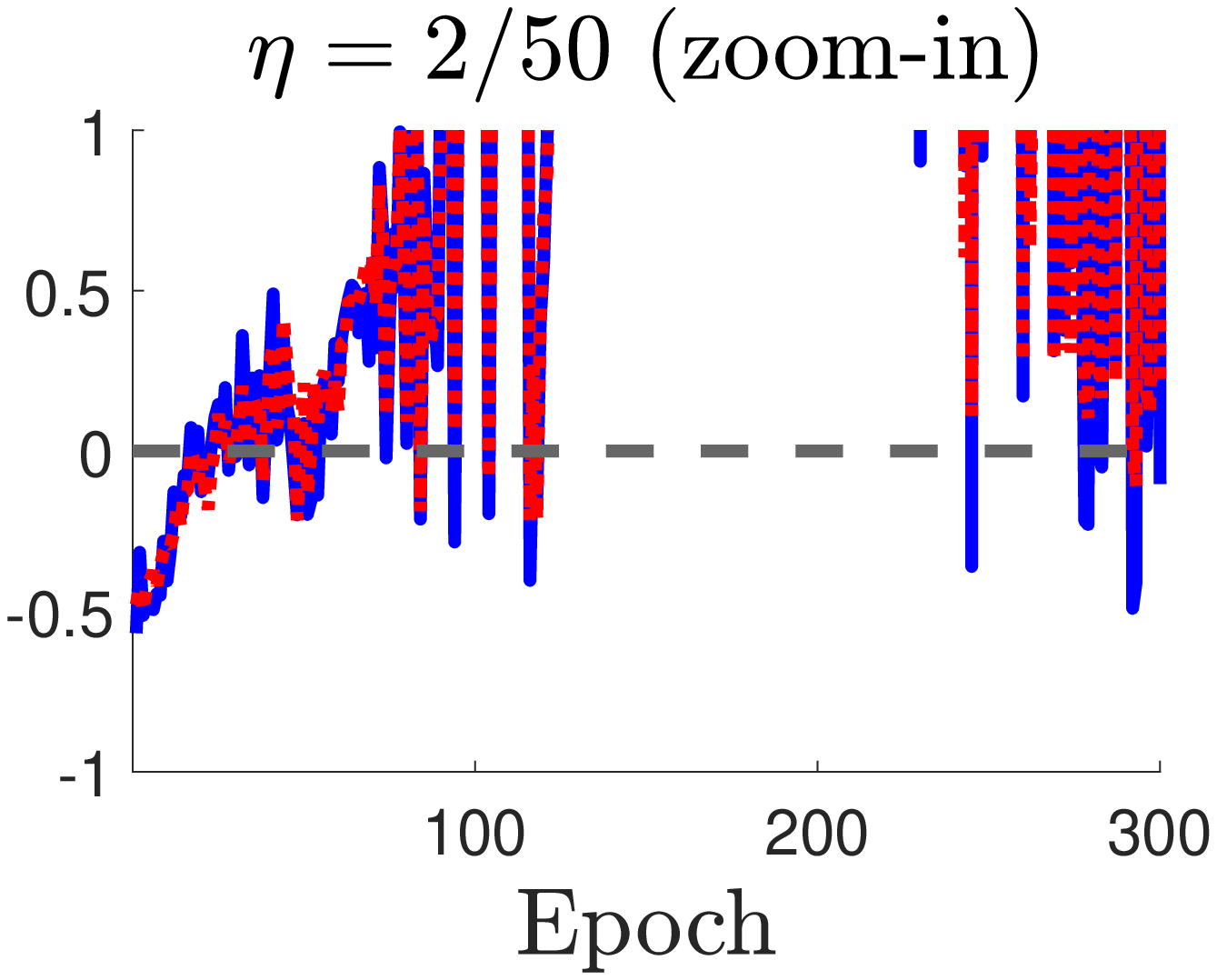}
\includegraphics[width=0.32\columnwidth]{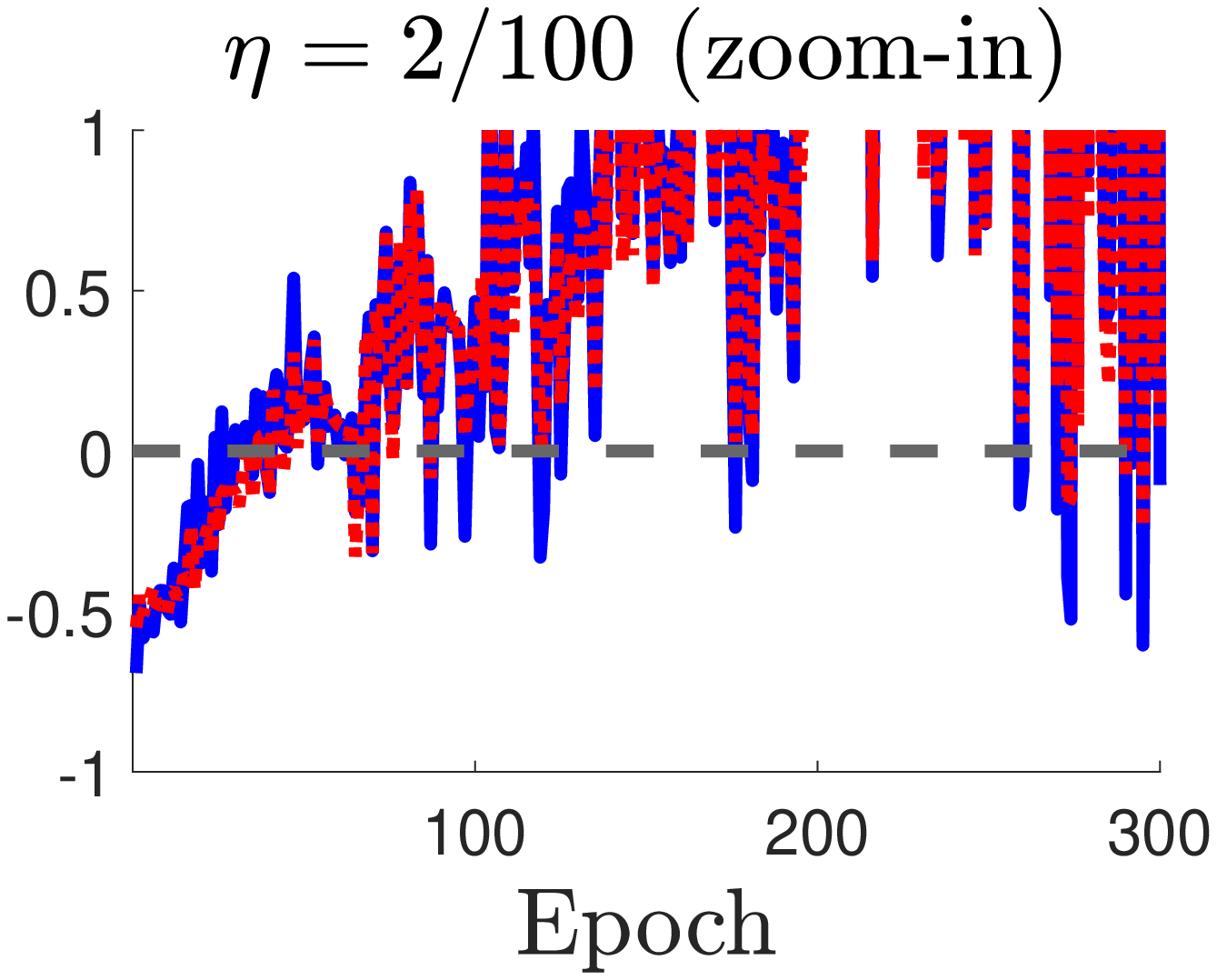}
\includegraphics[width=0.32\columnwidth]{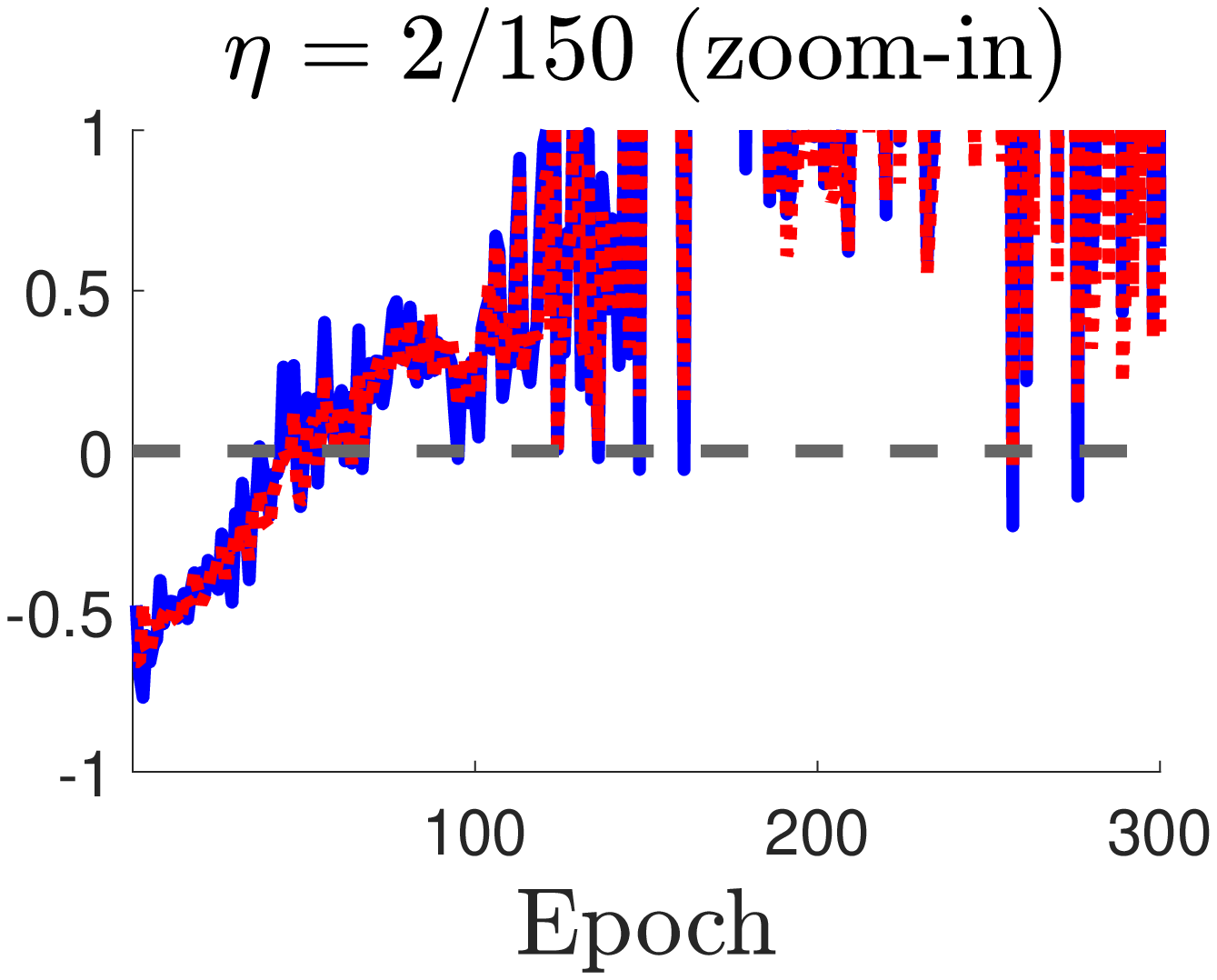}

 }
 
The results are similar to those for ReLU networks.\qqq
\end{experiment}

\section{Discussion and Future Directions}

This work demonstrated characteristics of unstable convergence that are in stark contrast with those of stable convergence.
Consequently, this work leads to several interesting future directions.\vspace{-5pt}
\begin{list}{{\tiny $\blacksquare$}}{\leftmargin=1.5em}   
  \item \emph{Better optimizer?} The characteristics based on the directional smoothness suggests that the adjacent iterates have nearly opposite gradients in the unstable regime. 
  This obviates the efficacy of many efficient methods (e.g. variance reduced methods) which are designed based on the intuition that the adjacent iterates have similar gradients.
  Hence, it would be interesting to design an efficient optimizer that respects the new characteristics.
  
  \item \emph{Faster training under unstable convergence?} As discussed in 
  \autoref{rmk:fast}, another striking feature of unstable convergence is that the training loss seems to converge faster.
  One interesting question is whether one can elucidate this faster optimization by further exploring our characterization of  relative progress.
  
  \item \emph{What assumptions are valid for neural network optimization?} It is clear that unstable convergence cannot be reasoned with the widespread condition of $\eta < \frac{2}{L}$.
  Then what assumptions would be valid for neural networks? 
  As discussed in \autoref{rmk:hessian_lip}, our \autoref{exp:const} suggests that  although the gradient Lipschitzness is not a good assumption for neural networks, some form of Hessian Lipschitzness might be a valid one. 

\item \emph{Unstable regime for adaptive methods?}
Our characterizations are limited to constant step size (S)GD, and it is not clear how these characterizations carry over to adaptive methods such as Adam and RMSProp. Investigating adaptive methods will help us understand how they differ from (S)GD for neural network optmization.
  \end{list}

\section*{Acknowledgements}
Kwangjun Ahn thanks {\bf Pourya Habib Zadeh} for various discussions and his help with experiments during the initial stage of this work.
Kwangjun Ahn also thanks {\bf Stephen J. Wright} and {\bf Sinho Chewi} for helpful discussions  while visiting the Simons Institute for the Theory of Computing.
Kwangjun Ahn also thanks {\bf Daniel Soudry} for fruitful discussions regarding \autoref{thm:master}.

Kwangjun Ahn and Suvrit Sra acknowledge support from an NSF CAREER grant (1846088), and NSF CCF-2112665 (TILOS AI Research Institute).
Kwangjun Ahn also acknowledges support from Kwanjeong Educational Foundation.
Jingzhao Zhang acknowledges support from IIIS young scholar fellowship.

\bibliographystyle{abbrvnat} 
\bibliography{ref}

\onecolumn
\appendix

\section{Proof of \autoref{thm:equiv}}
\label{app:pf:equiv}
 Using the fact
\begin{align*}
    \frac{\D}{\D \tau} \left[ f(\zz{} - \eta \tau   \nabla f(\zz{} )\right] = \inp{-\eta \nabla f(\zz{})}{\nabla f\big(\zz{} - \eta \tau \nabla f(\zz{})\big)}\,,
\end{align*}
we obtain 
\begin{align*}
  f\big(\zz{} -\eta \nabla f(\zz{})\big) - f(\zz{})  &= -\eta  \int_{0}^1 \inp{\nabla f(\zz{})}{ \nabla f\big(\zz{} - \eta \tau \nabla f(\zz{}) \big)} \D \tau\\
  &=-\eta \norm{\nabla f(\zz{})}^2-\eta  \int_{0}^1 \inp{\nabla f(\zz{})}{ \nabla f\big(\zz{} - \eta \tau \nabla f(\zz{}) \big)- \nabla f(\zz{})} \D \tau\,.
\end{align*}
Hence, after rearranging, we get
\begin{align*}
  \frac{\eta}{2} \cdot 2\int_{0}^1 \tau \cdot \dir{\eta \tau \nabla f(\zz{})}{\zz{}} \ \D \tau - 1=   \frac{f\big(\zz{} -\eta \nabla f(\zz{})\big) - f(\zz{})}{\eta \cdot \norm{\nabla f(\zz{})}^2 }   =   \rr{\zz{}}\,,
\end{align*}
which is precisely the relation in \autoref{thm:equiv}.

\paragraph{Derivation for the SGD case.} 
For the SGD case, the derivation is similar.
\begin{align*}
  & f\big(\zz{} -\eta g(\zz{})\big) - f(\zz{})
 \\
 &\quad  \overset{(a)}{=}     -\eta \int_{0}^1 \inp{g(\zz{})}{ \nabla f\big(\zz{} - \eta \tau g(\zz{}) \big)} \D \tau   \\
		&\quad  =     -\eta \inp{g(\zz{})}{\nabla f(\zz{})} -\eta  \int_{0}^1 \inp{g(\zz{})}{ \nabla f\big(\zz{} - \eta \tau g(\zz{}) \big)-\nabla f(\zz{})}   \D \tau
\end{align*}
where $(a)$ is due to the fact
\begin{align*}
    \frac{\D}{\D \tau} \left[ f(\zz{} - \eta \tau   g(\zz{} )\right] = \inp{-\eta g(\zz{})}{\nabla f\big(\zz{} - \eta \tau g(\zz{})\big)}\,.
\end{align*}
After taking expectation over the randomness in the stochastic gradient,  we obtain 
\begin{align*}
  & \E f\big(\zz{} -\eta g(\zz{})\big) - f(\zz{}) =     -\eta \norm{\nabla f(\zz{})}^2 -\eta  \int_{0}^1 \E \inp{g(\zz{})}{ \nabla f\big(\zz{} - \eta \tau g(\zz{}) \big)-\nabla f(\zz{})}   \D \tau
\end{align*}
Hence, after rearranging we obtain the desired equation:
\begin{align*}
 & \frac{\E f\big(\zz{} -\eta g(\zz{})\big) - f(\zz{})}{\eta \norm{\nabla f(\zz{})}^2} =     -1 +  \frac{\eta}{2} \cdot 2 \int_{0}^1 \tau\cdot \E_g\left[  \frac{\norm{g(\zz{})}^2 \cdot  \dir{\eta \tau g(\zz{})}{\zz{}}}{\norm{\nabla f(\zz{})}^2}  \right]  \  \D \tau \,.
\end{align*} 
This completes the derivation.

\end{document}